\documentclass[12pt,twosided]{scrartcl}

\usepackage[utf8]{inputenc}
\usepackage{graphicx}
\usepackage[all]{xy}
\usepackage{amssymb}
\usepackage{amsmath}
\usepackage{verbatim}
\usepackage{latexsym}
\usepackage{mathrsfs}
\usepackage{enumerate}
\usepackage{amsthm}

\makeindex

\newcommand{\aval}[1]{\left] #1 \right[} % open interval
\newcommand{\sval}[1]{\left[ #1 \right]} % closed interval
\renewcommand{\Cup}{\bigcup}
\renewcommand{\Cap}{\bigcap}

\renewcommand{\b}{\beta}
\newcommand{\g}{\gamma}
\newcommand{\e}{\varepsilon}
\renewcommand{\d}{\delta}

\renewcommand{\l}{\lambda}

\newcommand{\s}{\sigma}

\newcommand{\es}{\varnothing}

\newcommand{\A}{\mathcal{A}}
\newcommand{\B}{\mathcal{B}}

\newcommand{\K}{\mathcal{K}}

\newcommand{\LO}{{\operatorname{LO}}}
\newcommand{\N}{\mathbb{N}} %natural numbers
\newcommand{\Q}{\mathbb{Q}} %rationals
\newcommand{\R}{\mathbb{R}} %reals
\newcommand{\C}{\mathbb{C}}

\newcommand{\Hom}{{\operatorname{Hom}}}

\newcommand{\id}{\operatorname{id}}

\newcommand{\pr}{\operatorname{pr}}
\newcommand{\rest}{\!\restriction\!}
\newcommand{\dom}{\operatorname{dom}}
\renewcommand{\Im}{\operatorname{Im}}

\renewcommand{\le}{\leqslant} 
\renewcommand{\ge}{\geqslant} 
 
\DeclareMathOperator*{\comp}{\raisebox{0pt}{$\bigcirc$}}

\swapnumbers
\newtheorem*{Thm*}{Theorem}
\newtheorem{Thm}{Theorem}[section]
\newtheorem{Lemma}[Thm]{Lemma}

\newtheorem{Prop}[Thm]{Proposition}
\newtheorem{Fact}[Thm]{Fact}

\theoremstyle{definition}
\newtheorem{Def}[Thm]{Definition}

\newtheorem{Example}[Thm]{Example}

\theoremstyle{remark}
\newtheorem*{Remark}{Remark}

\author{Vadim Kulikov%
  \footnote{Also publishes as Vadim K. Weinstein.  Affiliations at the
    time of writing: KGRC, U. Vienna, Austria, and Department of
    Mathematics and Statistics, U. Helsinki, Finland. This paper
    has been published by the TAMS     http://dx.doi.org/10.1090/tran/6960
  }} \title{A Non-classification Result
  for Wild Knots}

\date{April 19, 2016}

\begin{document}

\maketitle

\begin{abstract}
  Using methods of descriptive theory 
  it is shown that the classification problem for wild knots is strictly harder
  than that for countable structures.
\end{abstract}

\section{Introduction}

A knot is a homeomorphic image of $S^1$ in $S^3$. Two knots $k$ and
$k'$ are equivalent if there exists a homeomorphism 
$h\colon S^3\to S^3$ 
such that $h[k]=k'$, where $f[A]$ denotes the image of the set
$A$ under~$f$. 
Denote this equivalence relation by $E_K$.
If a knot is equivalent to a finite polygon, then
it is \emph{tame} and otherwise \emph{wild}. There has been a
significant effort to classify various subclasses of wild knots: 
from early \cite{HarFox,DoHo,Lom, McP} till recent
\cite{JuLa, Fri, Nanyes}. Also higher dimensional wild knots have been
addressed~\cite{Fri1,Fri2}, and replacing $S^1$ by the Cantor set one
obtains the related study of the so-called wild Cantor sets whose classification
(up to an analogously defined equivalence relation)
has also been of interest because of their importance in dynamical systems~\cite{GRWZ}.

In this paper it is shown that $E_K$ does not admit classification by
countable structures while the isomorphism on the latter is
classifiable by~$E_K$.  More precisely we show that $E_K$ is strictly
above the isomorphism relation on countable structures in the Borel
reducibility hierarchy.  This is proved by Borel reducing both the
isomorphism on linear orders (which is sufficient by a theorem of
L. Stanley and H. Friedman, cf. Theorem~\ref{thm:FrSt} of this paper)
and a turbulent equivalence relation to $E_K$ (Theorems~\ref{thm:LO}
and~\ref{thm:Main}).

This implies in particular that wild knots cannot be completely
classified\footnote{That is, not in any \emph{reasonable} way,
  cf. \cite{Gao,Hjorth,Rosendal} for the theory of classification
  complexity. An example of a \emph{non-reasonable} classification
  would be to attach real numbers to the knot types using the Axiom of
  Choice.} 
by real numbers considered up to any Borel equivalence
relation (see remark after Theorem~\ref{thm:FrSt}).

\paragraph*{Acknowledgements}

I would like to thank Rami Luisto without whom the proof of Theorem~\ref{thm:Main}
would have been considerably more cumbersome, if existed at all. 
The research was supported
by the Austrian Science Fund (FWF) under project number P24654.

\section{Preliminaries}

In this Section we introduce the basic notions and theorems
that we need in the paper. 

\subsection*{Polish Spaces and Borel Reductions}

For any metric space $(X,d)$ we adopt the following notations:
\begin{eqnarray*}
B_X(x,\d)&=&\{y\in X\mid d(x,y)<\d\},\\
\bar B_X(x,\d)&=&\{y\in X\mid d(x,y)\le\d\},\\
B_X(A,\d)&=&\{y\in X\mid d(A,y)<\d\},\\
\bar B_X(A,\d)&=&\{y\in X\mid d(A,y)\le\d\},
\end{eqnarray*}
where $d(A,y)=\inf\{d(x,y)\mid x\in A\}$. We drop ``$X$'' from 
the lower case when it is obvious from the context.

\begin{Def}\label{def:Polish}
  A \emph{Polish} space is a separable topological space which is
  homeomorphic to a complete metric space.  The collection of Borel
  subsets of a Polish space is the smallest $\sigma$-algebra
  containing the basic open sets. A \emph{Polish} group
  is a topological group which is a Polish space and both the group
  operation and the inverse are continuous functions.

  A pair $(A,\B)$ is a \emph{standard Borel space} if $\B$ is a $\s$-algebra
  on $A$ and there exists a Polish topology on $A$ for which $\B$ is precisely
  the collection of Borel sets. 

  A function $f\colon A\to B$, where $A$ and $B$ are standard Borel
  spaces, is \emph{Borel}, if the inverse image of every Borel set
  is Borel.

  Suppose $A$ and $B$ are standard Borel spaces. Then an equivalence relation 
  $E\subset A\times A$ is \emph{Borel reducible} to an equivalence
  relation $E'\subset B\times B$, if there exists a Borel map 
  $f\colon A\to B$ such that for all $x,y\in A$ we have
  $$xEy\iff f(x)E'f(y).$$
  In this case we can also say that the elements of $A$ considered up to
  $E$-equivalence are classified in a Borel way by the elements of $B$ considered up
  to $E'$-equivalence, or just denote it by $E\le_B E'$. This is a quasiorder
  (transitive and reflexive) on equivalence relations.

  If $G$ is a Polish group which acts in a Borel way on a standard Borel space
  $B$ through an action $\pi$, then we denote by $E^{B}_{G,\pi}$ 
  the \emph{orbit equivalence relation} defined 
  by $$(x,y)\in E^{B}_{G,\pi}\iff \exists g\in G(\pi(g,x)=y).$$
  In this case we say that $E^B_{G,\pi}$ is \emph{induced} by the action
  of $G$ on $B$ and often drop ``$\pi$'' from the notation when it is
  either irrelevant or clear from the context. Such an equivalence relation is
  analytic as a subset of~$B\times B$.
\end{Def}

\begin{Fact}(\cite[Cor 13.4]{Kechris})\label{fact:stB} If $(A,\B)$ is a 
  standard Borel space and $B\in\B$, then $(B,\B\rest B)$ is 
  also a standard Borel space. \qed
\end{Fact}

\begin{Example}\label{ex:Polish}
  The complex numbers $\C$ is a Polish space and the unit circle $S^1\subset \C$ 
  is a Polish group whose group operation is the multiplication of complex numbers.
  The countable product $(S^1)^\N$ is a Polish group as well in the product topology.
  Given a compact space $X$, the group of self-homeomorphisms $\Hom(X)$
  is Polish in the topology induced by the sup-metric~\cite{Kechris}.

  Let $L$ be any countable (first-order) vocabulary and let $S(L)$ be
  the set of all countable $L$-structures with domain $\N$. This space
  can be viewed as a Polish space, in fact homeomorphic to the Cantor
  space $2^{\N}$. For instance, if $L$ contains just one binary
  relation $R$, then for each $\eta\in 2^{\N}$, let $\A_\eta$ be the
  model with domain $\N$ such that for all $x,y\in\N$,
  $R^{\A_\eta}(x,y)\iff \eta(\pi(x,y))=1$, where $\pi\colon
  \N\times\N\to\N$ is a bijection. 
  Denote by $\cong_{S(L)}$ the isomorphism relation
  $S(L)$ and more generally, if $M\subset S(L)$ is closed under isomorphism, 
  denote $\cong_{M}\ =\ \cong_{S(L)}\rest M$.
  It is not hard to see that $\cong_{S(L)}$
  is induced by an action of the infinite symmetric group $S_\infty$
  (cf.~\cite{Gao}) which is Polish as the subspace of the Baire space~$\N^\N$.

  Let $L=\{<\}$ be a vocabulary consisting of one binary relation and let
  $\LO\subset S(L)$ be the set of linear orders. It is a closed subset \cite{FrSt}
  and so Polish itself.
\end{Example}

\begin{Thm}[H. Friedman and L. Stanley \cite{FrSt}]\label{thm:FrSt}
  $\cong_{\LO}$ is $\le_B$-maximal among all isomorphism relations,
  i.e.  $\cong_{S(L)}\ \le_B \ \cong_{\LO}$ holds for any vocabulary~$L$. \qed
\end{Thm}

In particular $\cong_{\LO}$ is not Borel as a subset of $\LO\times \LO$, because the
isomorphism relation on graphs is complete analytic 
as a set \cite[Thm 4]{FrSt} and is Borel
reducible to $\cong_{\LO}$ by the above.
If an equivalence relation $E$ is reducible to $\cong_{S(L)}$ for some $L$ (or
equivalently to~$\cong_{\LO}$) we say, following \cite{Hjorth}, that $E$
\emph{admits classification by countable structures}.

\subsection*{Turbulence}

Turbulent group actions were introduced by G. Hjorth in \cite{Hjorth}.
The idea of turbulence is roughly to characterize those Polish group actions
whose orbit equivalence relation does
not admit classification by countable structures.

\begin{Def}[Hjorth]
  Suppose that $X$ is a Polish space
  and let $G$ be a Polish group acting continuously on~$X$.
  For $x\in X$ denote by $[x]=[x]_G$ the orbit of $x$.
  This action is said to be \emph{turbulent}, if the following conditions
  are satisfied:
  \begin{enumerate}
  \item every orbit is dense,
  \item every orbit is meager,
  \item for every $x,y\in X$,
    every open $U\subset X$ with $x\in U$ and every open $V\subset G$ with $1_G\in V$, 
    there exist a $y_0\in [y]_G$ and sequences $(g_i)_{i\in\N}$ 
    in $V$ and $(x_i)_{i\in\N}$ in $U$ such that 
    $y_0$ is an accumulation point of the set
    $\{x_i\mid i\in \N\}$ and $x_0=x$ and $x_{i+1}=g_ix_i$ for all $i\in\N$.
  \end{enumerate}
\end{Def}

Equivalence relations arising from turbulent actions do not admit classification
by countable structures~\cite[Cor 3.19]{Hjorth}:

\begin{Thm}[Hjorth]\label{thm:Hjorth}
  Suppose that a Polish group $G$ acts on a Polish space $X$ in a
  turbulent way. Then $E^X_G$ does not admit classification by
  countable structures, i.e.  $E^X_G\not \le_B\ \cong_{S(L)}$ for
  any~$L$. \qed
\end{Thm}

By Theorems~\ref{thm:FrSt} and~\ref{thm:Hjorth},
if one wants to show that some equivalence relation $E$ is strictly $\le_B$-above 
the isomorphism of countable structures, it
is sufficient to prove $\cong_\LO\ \le_B E$ and $E^*\le_B E$ for some turbulent~$E^*$.
This is our plan for $E=E_K$ In this case the turbulent equivalence relation $E^*$
will be a certain relation on~$(S^1)^\N$.

\begin{Def}\label{def:TheTurbEQ}
  For present purposes it will be convenient to redefine $S^1$ in a little bit
  non-standard way (still homeomorphic to the usual $S^1$ as a subset of $\R^2$). 
  Let $S^1=\sval{-\pi,\pi}/\{-\pi,\pi\}$, i.e. the closed interval 
  $\sval{-\pi,\pi}\subset \R$
  with the end-points identified. Denote by $S^1_{\C}$ the unit circle in the complex plane
  equipped with the group structure induced from multiplication in $\C$
  and let $\tau\colon S^1_\C\to S^1$ be the homeomorphism $\tau\colon e^{i\theta}\mapsto \theta$.
  This induces canonically a group structure on~$S^1$.

  Let $X=(S^1)^\N$ be the set of sequences on $S^1$ with the Tychonoff product topology. 
  Since $S^1$ is a topological group, this induces a group structure on~$X$.
  Let $G\subset X$ be the subgroup of those $\bar x$ for which
  $\lim_{n\to\infty}x_n=0$. This group is Polish when equipped with the sup-metric which
  induces a finer topology than that inherited from the ambient space, but gives the same Borel sets,
  i.e. the group is \emph{Polishable}.
  We say that two sequences $\bar x',\bar x\in (S^1)^\N$ are $E^*$-equivalent 
  if $\bar x'-\bar x\in G$. Thus $E^*$ is induced by the action of $G$ 
  given by translation. If one replaces $S^1$ by $\R$ in the above, then this equivalence
  relation is proved to be turbulent in \cite[§ 3.3]{Hjorth}. 
  The same proof works in this case as well;
  the only difference is that in his proof Hjorth uses division of reals by natural numbers
  and taking their absolute value. The only problem is the element 
  $\bar \pi=\{-\pi,\pi\}$ of $S^1$, so
  for the sake of the proof one can define $\bar \pi /k=\pi/k$ for all $k\ge 2$
  and $|\bar \pi|=\pi$ and otherwise simply compute everything as in~$\R$. 
  The essential properties
  that $|x/k|\stackrel{k\to\infty}{\longrightarrow}0$
  and $(\underbrace{x/k+\dots+x/k}_{k\ times})=x$ are then preserved.
\end{Def}

\subsection*{The Space of Knots}

We think of $S^3$ as the one-point compactification $\R^3\cup \{\infty\}$. 
The unit circle $S^1$ is parametrized as described in Definition~\ref{def:TheTurbEQ},
but this parametrization is relevant only in Section~\ref{sec:Turb}.
Let $K(S^3)$ be the space of compact subsets of $S^3$ equipped with the
Hausdorff metric:
\begin{equation}
  d_{K(S^3)}(K_0,K_1)=\inf\{\e\mid K_0\subset B(K_1,\e)\land K_1\subset B(K_0,\e)\}.\label{eq:HD}
\end{equation}

This space is compact Polish~\cite{Kechris}.
Let $\K$ be the subset of $K(S^3)$ consisting of knots, i.e. 
homeomorphic images of~$S^1$. The following is a consequence of a more general
theorem by Ryll-Nardzewski~\cite{RN}:

\begin{Thm}
  $\K$ is a Borel subset of $K(S^3)$ and so by Fact~\ref{fact:stB}
  it is a standard Borel space. \qed
\end{Thm}

\begin{Def}
  Suppose $X,Y$ are a topological spaces and $A\subset X, B\subset Y$ are
  subsets. We say that the pairs $(X,A)$ and $(Y,B)$ are homeomorphic (as pairs)
  if there exists a homeomorphism $h\colon X\to Y$ with $h[A]=B$.
  Thus, a knot can be thought of as a pair $(S^3,K)$ where $K$ is homeomorphic to 
  $S^1$. Two knots $K$ and $K'$ are equivalent, if $(S^3,K)$ and $(S^3,K')$
  are homeomorphic as pairs. This is the equivalence relation on 
  $\K$ denoted by $E_K$.

  $E_K$ is induced by a Polish group action in a natural way.
  Let $\Hom(S^3)$ be the group of homeomorphisms of $S^3$ onto itself.
  It is a Polish group (cf. Example~\ref{ex:Polish}) and
  it acts on $\K$ by $h\cdot K=h[K]$. 
\end{Def}

\begin{Def}\label{def:PropArc}
  Let $I$ be a set homeomorphic to the unit interval $\sval{0,1}$
  and let $\bar B$ be a set homeomorphic to a closed ball in $S^3$.  We say that 
  $f\colon I\to \bar B$ or $(\bar B, \Im f)$ or $(\bar B, f)$ 
  is a \emph{proper arc in $\bar B$} 
  if $x\in \partial I\iff f(x)\in \partial \bar B$. 
  Two proper arcs $f\colon I\to \bar B$,
  $f'\colon I\to \bar B'$
  are equivalent if $(\bar B,\Im f)$ and $(\bar B',\Im f')$
  are homeomorphic as pairs. In the case of knots and proper arcs we often abuse notation
  by denoting $\Im f$ by $f$. A
  proper arc is \emph{tame} if it is equivalent to a smooth arc or equivalently a 
  finite polygon. It is \emph{unknotted} or \emph{trivial} if it is equivalent to
  $$(\bar B_{\R^3}(\bar 0,1),\sval{-1,1}\times \{(0,0)\}).$$
  
  We say that a proper arc $(\bar B,f)$ has $(\bar B',g)$ 
  as \emph{as a component}, if there exists 
  a proper arc $(\bar B_0,h)$ and $\bar B_1\subset \bar B_0$
  such that $(\bar B,f)$ is equivalent to $(\bar B_0,h)$,  $(\bar B_1,h\cap\bar B_1)$ is a
  proper arc and $(\bar B',g)$ is equivalent to $(\bar B_1,h\cap\bar B_1)$.
  This is clearly independent of the choice of representatives.
\end{Def}

\section{Embedding $\cong_{\LO}$ into $E_K$}\label{ssec:MintoD}

\begin{Thm}\label{thm:LO}
  There is a Borel reduction $\cong_\LO\ \le_B \ E_K$.
\end{Thm}
\begin{proof}
  The idea is the following. 
  A linear order $L$ is embedded into the unit interval $I$ 
  such that every point in the image is isolated. Then the knot is constructed
  by replacing a small 
  neighborhood of each such point by a singular knot 
  depicted on Figure~\ref{fig:SingularKnot}. Now the
  order type of the isolated singularities of the resulting knot is the same as
  the order type of~$L$.
  All the technical details below are there to ensure that this function is
  indeed a Borel reduction of $\cong_{\LO}$ into~$E_K$.

  Similarly as in Definition~\ref{ex:Polish}, let $\LO$ be the set of those
  $R\in 2^{\N\times\N}$ which define a linear order on $\N$.
  Let $\LO^*$ be the space of those $(R,S)\in
  (2^{\N\times\N})^2$ for which $R$ defines a linear order on $\N$ with $0$ being the smallest
  element and with no greatest element and
  $S(n,m)=1$ if and only if there are no elements $R$-between $n$ and
  $m$.  Since $S$ is first-order definable from $R$, there is a canonical Borel map
  $\LO\to \LO^*$ which is a reduction from $\cong_{\LO}$ to $\cong_{\LO^*}$: For a linear
  order $L$, first let $L'=1+L+1+1+\Q$ and then canonically define $(R,S)$ so that $R$
  is $L'$. Obviously this is a Borel reduction and so
  it is sufficient to reduce $\cong_{\LO^*}$ into $E_K$.
  Suppose $L=(R_L,S_L)$ is an element of $\LO^*$, where we also denote $R_L=\,<_L$.
  Let
  $f_L\colon \N\to\R$ be a map and $(V^L_n)$ a sequence of open intervals
  defined by induction as follows. 
  First $f_L(0)=1/4$ and $V^L_0=\aval{0,\frac{1}{2}}$.
  Suppose $f_L(m)$ and $V^L_m$ are defined for $m<n$ such that for all distinct $m',m<n$
  \begin{itemize}
  \item $f_L(m)$ is the middle point of $V^L_m$, 
  \item $m'<_L m\iff f_L(m')< f_L(m)$,
  \item $V^L_{m'}\cap V^L_m=\es$,
  \item $\sup V^L_m<1$,
  \item if $m'<_L m$, then $\inf V^{L}_m - \sup V^{L}_{m'}=0\iff (m',m)\in S_L$.
  \end{itemize}
  There are two cases: either $m<_L n$ for all $m<n$
  or there are $m',m<n$ such that $m'<_L n<_L m$ and there is no
  $k<n$ which is $<_L$-between $m'$ and $m$. 
  In the first case let $m_0$ be the $<_L$-largest among all $m<n$. 
  Let $f_L(n)=\sup V^L_{m_0}+\frac{1}{3}(1-\sup V^L_{m_0})$ and if
  $(m_0, n)\in S$, then let $V^L_n=B(f_L(n),f_L(n)-\sup V^L_{m_0})$
  and otherwise let $V^L_n=B(f_L(n),\frac{1}{2}(f_L(n)-\sup V^L_{m_0}))$.
  In the second case define
  $$f_L(n)=
  \begin{cases}
    \frac{1}{2}(\sup V^L_{m'}+\inf V^L_m) &\text{if }(m',n)\in S\iff (n,m)\in S_L\\
    \sup V^L_{m'}+\frac{1}{3}(\inf V^L_m-\sup V^L_{m'}) &\text{if }(m',n)\in S_L\text{ and } (n,m)\notin S_L\\
    \sup V^L_{m'}+\frac{2}{3}(\inf V^L_m-\sup V^L_{m'}) &\text{if }(m',n)\notin S_L\text{ and } (n,m)\in S_L.
  \end{cases}
  $$
  Then define
  $$V^L_n=
  \begin{cases}
    B(f_L(n),\frac{1}{2}(\inf V^L_{m}-\sup V^L_{m'})) &\text{ if }(m',n)\in S_L \text{ and }(n,m)\in S_L\\
    B(f_L(n),\frac{1}{4}(\inf V^L_{m}-\sup V^L_{m'})) &\text{ if }(m',n)\notin S_L \text{ and }(n,m)\notin S_L\\
    B(f_L(n),f_L(n)-\sup V^L_{m'}) &\text{ if }(m',n)\in S_L \text{ and }(n,m)\notin S_L\\
    B(f_L(n),\inf V^L_m - f_L(n)) &\text{ if }(m',n)\notin S_L \text{ and }(n,m)\in S_L.
  \end{cases}
  $$

  This ensures the following:
  \begin{enumerate}
  \item for every distinct $n,m\in \N$ the sets $V^L_n, V^L_m$ are disjoint and $
    f_L(n)$ is the middle point of $V_n^L$, 
  \item $\Cup_{n\in \N}V^L_n$ is dense in $\sval{0,1}$,
  \item if $L\rest \{0,\dots,n\}=L'\rest \{0,\dots,n\}$, then
    $f^L\rest \{0,\dots,n\}=f^{L'}\rest \{0,\dots,n\}$ and
    $V^L_m=V^{L'}_m$ for all $m\le n$. \label{thingyything}
  \item $\sup V^L_n = \inf V^L_m\iff n<_L m \text{ and }(n,m)\in S_L$
  \item $f_L(m')<f_L(m)\iff m'<_L m$.
  \end{enumerate}
  Now pick open intervals $U^L_n\subset V^L_n$ such that $\overline{U_n^L}\subset V^L_n$
  and $f_L(n)$ is the center of $U^L_n$.
  Thinking of $\R$ as the $x$-axis of $\R^3$, replace each $U^L_n$ by the open ball
  $B^L_n$ with the center at $f_L(n)$ and radius $\frac{1}{2}|U^L_n|$.
  Now build the knot $K=F(L)\subset S^3$ as follows. Let $K_a^L=(\R\setminus \sval{0,1})\cup\{\infty\}$. 
  For each $n\in \N$ let $K^L_n\subset \overline{B^L_n}$ be a 
  proper arc in $\overline{B^L_n}$ with end-points
  coinciding with the end-points of $U^L_n$.
  This arc $K^L_n$ is equivalent to the knot depicted on Figure~\ref{fig:SingularKnot},
  it is a knot sum of infinitely many trefoils
  and has one singularity precisely at $f_L(n)$.
  Let $K^L_r$ cover the rest: $K^L_r=\sval{0,1}\setminus \Cup_{n\in \N}U^L_n$.
  Now define
  $$K=F(L)=\Cup_{x\in \{a,r\}\cup L}K^L_x.$$
  Let us show that $F$ is continuous. Let $L\in\LO^*$. We will show
  that for every open neighborhood $U$ of $F(L)$ there is an open neighborhood of
  $L$ which is mapped inside $U$. Suppose that $U$ is the $\e$-ball with center $F(L)$
  in the Hausdorff metric, cf.~\ref{eq:HD}. Pick $n$ so large that $|V^L_n|<\e/2$. We claim that
  for any $L'$ with $L'\rest \{0,\dots,n\}=L\rest \{0,\dots,n\}$ we have
  $F(L')\subset B(F(L),\e)$ and $F(L)\subset B(F(L'),\e)$. 
  The argument is symmetric, so we will just show the first part.
  If $x$ is in $F(L')$, then it is in one of
  the sets $K^{L'}_z$ for $z\in \{a,r\}\cup \N$. 
  If $z\in \{a,0,\dots,n\}$, then by \ref{thingyything}, 
  $x\in F(L)$. Suppose $x\in K^{L'}_m$ for some $m>n$. Denote by $\pr x$
  the projection of $x$ onto the $x$-axis, so we have
  $d(x,\pr x)<\e/2$. On the other hand $\pr x\notin \bar B^{L}_{k}$
  for any $k\le n$, so $d(\pr x,F(L))<\e/2$ and we have $x\in B(F(L),\e)$. 

  \begin{figure}
    \centering
    \includegraphics[width=\textwidth]{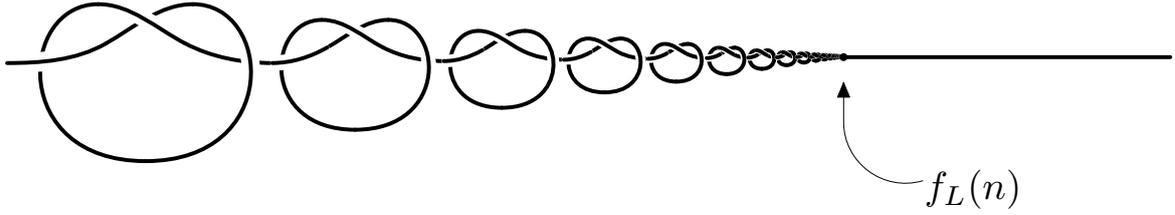}
    \caption{The singular arc whose copies are concatenated to form the knot $F(L)$. Figure produced by MetaPost.}
    \label{fig:SingularKnot}
  \end{figure}

  Suppose $L\cong L'$ are isomorphic structures in $\LO^*$
  and let $g\colon \N\to \N$ be a
  witnessing isomorphism. It induces an order preserving map $h_0$ from
  $f_{L}[L]$ onto $f_{L'}[L']$ defined by $f_L(n)\mapsto f_L(g(n))$. Let 
  $$W_n^L=\overline{V^L_n}\times \R^2$$
  and 
  $$W_n^{L'}=\overline{V^{L'}_n}\times \R^2.$$
  Call these sets ``walls''.
  Restricted to $W^L_n$, the domain of $h_0$ is just the singleton $\{f_L(n)\}$ which is
  mapped to $\{f_{L'}(g(n))\}$. Let 
  $h_1^n\colon \overline{B^L_n}\to \overline{B^{L'}_{g(n)}}$ 
  be a homeomorphism such that
  $h_1(f_L(n))=f_{L'}(g(n))=h_0(f_L(n))$, which
  takes $K^{L}_n$ to $K^{L'}_{g(n)}$. 
  This is possible because these arcs are equivalent.
  Let $M_n=\frac{1}{2}\max\{|V^L_n|,|V^{L'}_{g(n)}|\}$ and let 
  $$h_2^n\colon \overline{V^L_n}\times \sval{-M_n,M_n}^2\to \overline{V^{L'}_{g(n)}}\times \sval{-M_n,M_n}^2$$
  be a homeomorphism
  which extends $h_1^n$ and with the property that for all $x\in \sval{-M_n,M_n}^2$ and
  $z_1, z_2, z_1',z_2'$ the left and right endpoints of $V^L_n$ and $V^{L'}_{g(n)}$ respectively we have
  $$h^n_2(z_i,x)=(z_i',x).$$
  Finally define a homeomorphism $h_3^n\colon W^L_n\to W^{L'}_{g(n)}$
  which extends $h_2^n$
  such that
  \begin{enumerate}
  \item[] For every $(x,y_1,y_2)\in V^L_n\times \R^2$ 
    with $\max\{y_1,y_2\}>M_n$\\
    we have $h_3(x,y_1,y_2)=(x',y_1,y_2)$ for some $x'$. \hfill $(P)$
  \end{enumerate}
  Consider a pair $(n,m)$ such that $(n,m)\in S_L$ and $n<_L m$. Then
  by the definitions above for all $x\in W^L_n\cap W^L_m$ we will have
  $$h^n_3(x)=h^m_3(x)\eqno (Q)$$
  Let $h_a$ be the identity on $\R^3\setminus
  \sval{0,1}\times \R^2$.  Consider $x$ which is a cluster point of
  $\Im f_L$. Let $x'$ be the corresponding point in $\Im f_L$. 
  Define $h_r(x,y_1,y_2)=(x',y_1,y_2)$ for all $(y_1,y_2)\in \R^2$.
  By $(P)$ and $(Q)$ above and noting that $M_n\stackrel{k\to\infty}{\longrightarrow}0$, the union
  $$h=h_a\cup h_r\cup \Cup_{n\in\N}h_3^n$$ 
  is a well defined function and is a homeomorphism.

  On the other hand suppose there exists a
  homeomorphism $h$ which takes $F(L)$ to $F(L')$. Without loss of generality we
  can assume that $h(\infty)=\infty$, i.e. $h$ restricts to a homeomorphism of $\R^3$.
  Let $\Sigma_{L}$ be
  $\{x\in F(L)\mid x\text{ is a singularity}\}$, $I\Sigma_L$ be
  the set of isolated points in $\Sigma_L$ and similarly define $\Sigma_{L'}$ and $I\Sigma_{L'}$
  for~$L'$. 
  A singular point can only go to a singular point in a homeomorphism
  because all other points are locally unknotted but every neighborhood of
  a singular point contains trefoils as components (Definition~\ref{def:PropArc}). Hence
  $h\rest \Sigma_L$ is a homeomorphism onto~$\Sigma_{L'}$.
  $I\Sigma_L$ is a subset of the $x$-axis which we identified with $\R$,
  so we can define the order $<_\R$ on $I\Sigma_{L}$ and by \ref{thingyything}
  $f_n(L)$ is an order isomorphism from $(\N,<_L)$ to $(I\Sigma_L,<_\R)$ and 
  same for~$L'$.
  Being an isolated point is a topological property, so $h\rest I\Sigma_L$ is in fact
  a homeomorphism onto $h\rest I\Sigma_{L'}$. Also, it has to preserve the betweenness relation:
  suppose $x,y,z$ are points in $I\Sigma_{L'}$ such that $x<_{\R}y<_{\R}z$.
  Now $x$ and $z$ divide the knot into two arcs one of which is bounded in $\R^3$
  and $y$ is in this bounded arc. If $h$ mapped $y$ such that it is not $<_\R$-between
  $h(x)$ and $h(z)$, then it would mean that $h(y)$ is mapped to the unbounded component of
  $F(L')\setminus \{h(x),h(z)\}$ which is a contradiction.
  This implies that $(f_{L'})^{-1}\circ (h\rest I\Sigma_L)\circ f_L$ 
  is a bijection which is either order-preserving or order-reversing (because it preserves
  betweenness).
  But since we assumed that the elements of $\LO^*$ have a smallest element but no largest,
  mirror images are out of the question, so it is in fact an isomorphism.
\end{proof}

\section{Reducing a Turbulent Relation into~$E_K$}\label{sec:Turb}

Recall the equivalence relation $E^*$ from Definition~\ref{def:TheTurbEQ}.
The proof of the following theorem can be found after all
the preparations on page~\pageref{page:proofofmain}.

\begin{Thm}\label{thm:Main}
  The equivalence relation $E^*$ on $(S^1)^{\N}$
  (cf. Definition~\ref{def:TheTurbEQ}) is continuously reducible
  to~$E_K$.
\end{Thm}

Let us begin with some preliminaries in topology. 
For metric spaces $X$ and $Y$, where $Y$ compact, let
$C(X,Y)$ be the space of all continuous maps from $X$ to $Y$ equipped
with the sup-metric. 
A map $f\colon X\to Y$ is $L$-\emph{Lipschitz}, for some $L>0$, if for all $x,y\in X$
$$d_Y(f(x),f(y))\le Ld_X(x,y).$$
It is $L$-\emph{colipschitz}, if for all $x,y\in X$
$$d_Y(f(x),f(y)) \ge \frac{1}{L}d_X(x,y).$$
It is $L$-bilipschitz, if it is both $L$-Lipschitz and $L$-colipschitz.
Note that the notion of $L$-bilipschitz makes sense only when $L\ge 1$.
Also note that if $f$ is $L$-colipschitz, then $f$ has an inverse 
and this inverse is $L$-Lipschitz. Thus, one can say that a function is $L$-bilipschitz
if both it and its inverse are $L$-Lipschitz.

A map between metric spaces is \emph{locally $L$-Lipschitz} if
for every point in the domain there exists $\e$ so that the function is
$L$-Lipschitz restricted to the $\e$-neighborhood of this point.
Analogously define \emph{locally $L$-colipschitz} and \emph{locally $L$-bilipschitz}.
The length of a path in a metric space is defined through the supremum of the
lengths of finite polygons approximating the path. Assuming that there are 
paths of finite length between any two given points (call such spaces \emph{length-spaces}),
one can define the \emph{path metric} $d^p$ given by the infimum of all path lengths
from one point to another. All spaces in this paper are length spaces.
The following is a standard theorem of metric topology.

\begin{Prop}\label{prop:BilipPathMetric}
  Suppose $f\colon X\to Y$ is a locally $L$-Lipschitz homeomorphism for length spaces $X,Y$.
  Then $f$ is $L$-Lipschitz as a function from
  from $(X,d^p_X)$ to $(Y,d^p_Y)$ where $d^p_X$ and 
  $d^p_Y$ are path metrics on $X$ and $Y$ respectively. Same for ``Lipschitz'' replaced
  by ``colipschitz''. \qed
\end{Prop}

Denote by $D^2_r=\bar B_{\R^2}(0,r)$ the two dimensional closed
disk with radius $r$. Let $T_r$ be the solid torus
$T_r=D^2_r\times S^1$. The metric on this torus is the product metric:
$d_{T_r}((x,s),(y,t))=\sqrt{d_{\R^2}(x,y)^2 + d_{S^1}(x,y)^2}$. Note that
from our parametrization of $S^1$ (Definition~\ref{def:TheTurbEQ})
this metric coincides with the path metric on~$T_r$, i.e. it 
is in a sense ``convex''.

We define a tubular neighborhood in a little bit non-standard way.
For a proper arc $f\colon I\to \bar B_0$, as in Definition~\ref{def:PropArc},
the set $T(f,\e)=\bar B_{S^3}(f[I],\e)\cap \bar B_0$ is a \emph{tubular neighborhood} of $f$
if there exists an embedding
$h\colon D^2_\e\times I\to S^3$
such that 
\begin{itemize}
\item $\Im h \cap \bar B_0=T(f,\e)$,
\item $h(\bar 0,s)=f(s)$ for all $s\in I$,
\item for all $\e'<\e$ and $x\in \Im(f)$ the set $B(x,\e')\cap f[I]$
  is connected and represents an unknotted proper arc in~$B(x,\e)$,
\item $h\rest D^2_r\times \{s\}$ is an isometry onto the disc $D\subset \R^3$
  which is orthogonal to $f$ at $f(s)$ and whose middle point is $f(s)$.
\end{itemize}
Similarly define this for a knot $f\colon S^1\to \R^3$ just replacing $I$
with $S^1$.
The following is a standard fact in knot theory and differential topology in 
general:

\begin{Thm}\label{thm:Tub}
  For every smooth proper arc or a knot
  $f$ there is $\e>0$ such that $T(f,\e)$ is a
  tubular neighborhood of~$f$. 
\end{Thm}
\begin{proof}
  Let us prove this for a knot $f\colon S^1\to \R^3$, so the result for a proper arc
  follows: every proper arc can be extended to a knot by connecting the end-points in the ambient space.
  We can assume without loss of generality that $|f'(s)|>0$ for all $s$.
  Let $G\subset S^2$ be the set of all the directions of the gradient:
  $$G=\big\{\frac{f'(s)}{|f'(s)|}\mid s\in S^1\big\}.$$
  Since $f$ is smooth and $S^1$ compact, $G$ is nowhere dense, so 
  there exists $s_0\in S^2\setminus G$. For every $s\in S^1$ let $M_s$ be the $2$-dimensional 
  subspace of $\R^3$ orthogonal to $f'(s)$ and let $n_s$ be the orthogonal projection of
  $s_0$ to $M_s$ which is normalized to length~$1$ (since $s_0\notin G$, the orthogonal projection
  is non-zero). Since $f$ is smooth, $s\mapsto n_s$ is also smooth and $n_s$ is normal to $f'(s)$.
  Let 
  $$b_s=\frac{f'(s)\times n_s}{|f'(s)\times n_s|}.$$
  Now define the following map:
  $$g\colon \R^2\times S^1\to \R^3,$$
  $$g(x,y,s)=xn_s+yb_s+f(s).$$
  By the smoothness of $f$, $s\mapsto n_s$ and $s\mapsto b_s$, $g$ is also smooth.
  By smoothness and compactness one finds $\e$ such that restricted to $D^2_\e\times S^1$,
  $g$ is injective. It is easy to see that all the conditions for tubular neighborhood
  are satisfied in particular because $f[S^1]$ is now a strong 
  deformation retract of $g[D^2_\e\times S^1]$,
  so restricted
  to $D^2_\e\times S^1$, $g$ witnesses that $T(f,\e)$ is a tubular neighborhood.
\end{proof}

If $\e$ is such that $T(f,\e)$ is a tubular
neighborhood, we say that $\e$ is \emph{tubular for}~$f$. It is easy to see that
if $\e$ is tubular for $f$, then so is every $\e'<\e$.

\begin{Lemma}\label{lemma:asacomponent}
  Suppose $f\colon I\to \bar B$ is a smooth proper arc in a closed
  ball $\bar B=\bar B(x,r)$.  Suppose $\e$ is tubular for $f$.  Then
  any tame proper arc $f_0\in B_{C(I,\bar B)}(f,\e)$ has $f$ as a
  component.
\end{Lemma}
\begin{proof}
  Suppose $f_0\in B_{C(I,\bar B)}(f,\e)$ is a tame proper arc. Without loss
  of generality $I=\sval{0,1}$ is the unit interval.
  Suppose $g\colon D^2_\e\times I\to S^3$ is a homeomorphism witnessing that
  $T(f,\e)$ is a tubular neighborhood. 
  There is a $\d_0$ such that 
  $$(D_\e^2\times \sval{0,\d_0}, \g_0),$$
  where $\g_0=\Im (g^{-1}\circ f_0)\cap D_\e^2\times \sval{0,\d_0}$,
  is unknotted.
  Let $\d_1$ be so small that 
  $$(\bar B(x,r-\d_1),f\cap \bar B(x,r-\d_1))$$
  is a connected proper arc equivalent to $(\bar B, f)$. This can be found again by
  compactness and smoothness. Let $\d$ be the minimum of $\d_0$
  and the following two:
  $$\sup \{\d_2\mid f[0,\d_2]\cap \bar B(x,r-\d_1)=\es\}.$$
  $$\sup \{\d_3\mid f[1-\d_3,1]\cap \bar B(x,r-\d_1)=\es\}.$$

  Let $\e'<\e$ be such that $g^{-1}\circ f_0\subset D^2_{\e'}\times I$ which exists by
  compactness and the fact that $g^{-1}\circ f_0$ is in the interior of $D^2_\e\times I$.
  Denote $\g=\Im (g^{-1}\circ f_0)\cap D_{\e'}^2\times \sval{0,\d}$
  Since now
  $$(D_{\e'}^2\times \sval{0,\d}, \g)$$
  is an unknotted proper arc, there is a homeomorphism of pairs
  $$H\colon (D_{\e'}^2\times \sval{0,\d},\g)\to (D_{\e'}^2\times \sval{0,1-\d},\{0,0\}\times \sval{0,1-\d}).$$
  Extend $H$ to $H_1\colon D_{\e}^2\times I\to D_{\e}^2\times I$ so that $H_1$ fixes 
  $(\partial D^2_\e)\times I$ point-wise. Now $g\circ H_1\circ g^{-1}$ can be extended to 
  a homeomorphism of $\bar B$ (by just identity outside $\Im g$) witnessing that 
  $f_0$ has $f$ as a component.
\end{proof}

For our purposes we have to redefine the notion of a component of a knot.
In Definition~\ref{def:PropArc} we gave a fairly standard definition of what that means.
However, it is not suitable for the kind of fractal knots we are dealing with in this
section. For example the knot $K(\bar x)$ that we are going to define is not
smooth at any point, so it will not contain any tame knot as a component.
On the other hand it is built as a limit of tame knots $K_n$ such that every prime knot
appears as a component of $K_n$ for sufficiently large~$n$. So we want to redefine the
notion of a component to apply to such wild knots in a more natural way.

\begin{Def}\label{def:asacomponent}
  Let $f\colon I\to \bar B$ be a proper arc (not necessarily tame)
  and $g$ another proper arc. We say that $f$ \emph{has $g$ strongly as a component}
  if there exists an $\e$ such that every tame proper arc
  $f_1\in B_{C(I,\bar B)}(f,\e)$ has $g$ as a component.
  By the above Lemma, if $f$ is smooth, then it has $g$ strongly as a component
  if and only if it has it as a component, and this is witnessed by any $\e$ which
  is tubular for~$f$.
\end{Def}

Having $g$ strongly as a component is an invariant property:

\begin{Lemma}\label{lemma:invarianceofcomp}
  If $f\colon I\to \bar B$ has $g$ strongly as a component and $H\colon \bar B\to \bar B'$ is
  a homeomorphism, then also $(\bar B',H\circ f)$ has $g$ strongly as a component.
\end{Lemma}
\begin{proof}
  Suppose $\e$ witnesses that $f$ has $g$ strongly as a component. $H$ induces a homeomorphism
  $\bar H$ of $C(I,\bar B)$, so there is $\e'$ such that 
  $$\bar H^{-1}[B_{C(I,\bar B)}(\bar H(f),\e')]\subset B_{C(I,\bar B)}(f,\e).$$
  Given any tame arc $h\in B_{C(I,\bar B)}(H\circ f,\e')$, it is equivalent to
  $\bar H^{-1}(h)=H^{-1}\circ h$ which is a tame arc in $B_{C(I,\bar B)}(f,\e)$
  and so by the definition of having $g$ strongly as a component, has $g$ as a component.
\end{proof}

\begin{Lemma}\label{lemma:HerComp}
  Let $f\colon I\to \bar B$ be a proper arc and suppose that it has
  $K$ strongly as a component which is witnessed by~$\e$.  Suppose
  $f'\colon I\to \bar B$ satisfies $d(f',f)<\e/2$.  Then $f'$ has $K$
  strongly as a component.
\end{Lemma}
\begin{proof}
  Trivial.
\end{proof}

Before the next Lemma, let us state a well known fact from differential geometry:

\begin{Fact}(Folklore)
  The operator norm of an $(n\times n)$-matrix $A$ is given by
  $$\sup\{|A\bar x| \colon |\bar x|=1,\bar x\in\R^n\}.$$
  A smooth function $f$ from $n$-manifold $N$ to $n$-manifold $M$ 
  is locally $(L+\e)$-Lipschitz
  at $x\in N$ for all $\e>0$ if and only if the operator norm of the Jacobian
  at $x$ is at most~$L$. \qed
\end{Fact}

\begin{Lemma}\label{lemma:TubularBilipschitz}
  If $f\colon S^1\to \R^3$ is a smooth curve which is locally $L$-bilipschitz
  for some $L\ge 1$, then for
  every $\e>0$ it has a tubular neighborhood which is realized by a locally
  $(L+\e)$-bilipschitz homeomorphism $g\colon D^2_r \times S^1\to T(f,r)$ for some~$r$.
\end{Lemma}
\begin{proof}
  Let $g$ be as in the proof of Theorem~\ref{thm:Tub}. 
  Denote $n_s=(n^1_s,n^2_s,n^3_s)$ and
  $b_s=(b^1_s,b^2_s,b^3_s)$ for all~$s$ and $f(s)=(f_1(s),f_2(s),f_3(s))$.
  The Jacobian of $g$ at $(0,0,s)$ is then
  $$J(s)=\left(
  \begin{matrix}
    n_s^1 & b_s^1 & f_1'(s)\\
    n_s^2 & b_s^2 & f_2'(s)\\
    n_s^3 & b_s^3 & f_3'(s)
  \end{matrix}
  \right).
  $$
  Let $\bar x\in \R^3$ be any unit vector, $|\bar x|=1$, $\bar x=(x_1,x_2,x_3)$.
  Then $$J(s)\bar x=x_1 n_s + x_2 b_s + x_3 f'(s).$$
  Since $L\ge 1$, $|f'(s)|\ge 1$, $|n_s|=|b_s|=1$ and all these vectors are orthogonal to
  each other, we have:
  \begin{eqnarray*}
    |J(s)\bar x|&=&|x_1 n_s + x_2 b_s + x_3 f'(s)|\\
    &=&\sqrt{x_1^2 + x_2^2 + x_3^2 |f'(s)|^2}\\
    &\le &\sqrt{x_1^2 |f'(s)|^2 + x_2^2 |f'(s)|^2 + x_3^2 |f'(s)|^2}\\
    &= & |f'(s)|\sqrt{x_1^2 + x_2^2 + x_3^2}\\
    &= & |f'(s)|.
  \end{eqnarray*}
  This means that the vector $\bar x=(0,0,1)$ maximizes the norm of
  $J(s)\bar x$ whence it is $ |f'(s)|$ and so this means that the operator norm
  of $J(s)$ is $|f'(s)|$. By the continuity of the Jacobian and compactness,
  we can find $r$ such that the Jacobian of $g$ at $(x,y,s)$ for all $(x,y)\in D^2_r$
  has operator norm at most $|f'(s)|+\e/2$. Then restricted to $D^2_r\times S^1$, $g$
  will be locally $(|f'(s)|+\e)$-Lipschitz. Consider then the inverse $J(s)^{-1}$.
  Let $\bar x$ be arbitrary unit vector in the coordinate system
  $(n_s,b_s,f'(s)/|f'(s)|)$ and consider
  \begin{eqnarray*}
    |J(s)^{-1}(x_1n_s+x_2b_s+x_3f'(s)/|f'(s)|)|&=&|J(s)^{-1}J(s)(x_1,x_2,x_3/|f'(s)|)|\\
    &=&|(x_1,x_2,x_3/|f'(s)|)|\\
    &\le&|(x_1,x_2,x_3)|\\
    &=&1
  \end{eqnarray*}
  where the last inequality follows from the assumption that $|f'(s)|\ge 1$.
  So the operator norm of $J(s)^{-1}$ is $1$.
  By moving to yet smaller $r$, one can ensure that the operator norm of $g^{-1}$ is everywhere
  at most $1+\e/2\le L+\e/2$. Thus, for this $r$, both $g$ and $g^{-1}$
  are locally $(L+\e)$-Lipschitz, and so $g$ is locally $(L+\e)$-bilipschitz.
\end{proof}

For a sequence of functions $(f_n)$ such that $\Im f_n\subset \dom f_{n+1}$
for all $n$, denote by $\comp_{k=i}^jf_k=f_i\circ\cdots \circ f_j$ 
their composition. The order matters in composition, so if $j<i$, this notation
simply means that the functions are taken in reverse order, for example one
could write
$$\left(\comp_{i=2}^5 f_i\right)^{-1}=\comp_{i=5}^2 f^{-1}_i.$$
Otherwise we often omit the circle from the composition notation, writing
just $f\circ g=fg$. We also sometimes omit the brackets from the function
notation, i.e. $f(x)=fx$.

Let $(K_n)_{n\in\N}$ enumerate all (tame) prime knot 
types (trivial knot not included).
The content of the following proposition is partially illustrated in Figure~\ref{fig:Embedding}.

\begin{Prop}\label{prop:SequencesOfEmbeddings}
  There exist sequences $(\e_n)$, $(\l_n)$, $(L_n)$, $(g_n)$ and $(f_n)$ such that
  for all~$n$, 
  \begin{enumerate}[(a)]
  \item $g_{n}\colon T_{\e_{n+1}}\to T_{\e_n}$, $f_n\colon S^1\to T_{\e_{n}}$
    and $g_n(\bar 0,s)=f_n(s)$ for all $s\in S^1$,
  \item $\e_0=1$, $L_n<2^{2^{-n}}$ and $4\l_{n+1}<2\e_{n+1}<\l_n<\e_n\le 2^{-n}$, \label{sequences(g)}
  \item $20\e_{n+1}$ is tubular for $f_n$, \label{sequences(h)}
  \item $\Im g_n=T(f_n,\e_{n+1})\subset T_{\l_n}$ and $g_n$ witnesses that $T(f_n,\e_{n+1})$ 
    is a tubular neighborhood of $f_n$,  \label{sequences(a)}
  \item $f_n(s)=s$ for $s\notin B_{S^1}(0,\l_n)$ and $f_n(s)\in B_{T_{\e_n}}(\bar 0,\l_n)$
    for $s\in B_{S^1}(0,\l_n)$, \label{sequences(b)}
  \item $(B_{T_{\e_n}}(\bar 0,\l_n),f_n\cap B_{T_{\e_n}}(\bar 0,\l_n))$  is a proper arc of knot type~$K_n$, \label{sequences(c)}
  \item $f_n$ is locally $1$-colipschitz, \label{sequences(d)}
  \item $g_n$ is locally $L_n$-bilipschitz, \label{sequences(e)}
  \item $g_n$ is $\e^2_{n}/\e_{n+1}$-colipschitz. \label{sequences(f)}
  \end{enumerate}
\end{Prop}
\begin{proof}
  Let $\e_0=1$, $\l_0=1/2$, $L_0=3/2$ and $f_0$ is a smooth map
  $S^1\to T_{\e_0}$ satisfying~\eqref{sequences(b)}
  and~\eqref{sequences(c)}. By making the knot sufficiently small, we
  can make sure that the length of the curve is arbitrarily close to
  the length of $S^1$, so that $f_0$ is locally $L_0'$-bilipschitz for some
  $L_0'<L_0$ and
  we can also make sure that the gradient is always at least $1$,
  i.e. $f_0$ is $1$-colipschitz, so \eqref{sequences(d)} is satisfied.
  As an induction hypothesis assume that $\e_m,\l_m,L_m,f_m$ are
  defined for all $0\le m\le n$ and $g_m$ is defined for all 
  $0\le m<n$. If $m<n$, then all conditions are satisfied. For $\e_n,\l_n,L_n,f_n$ 
  only \eqref{sequences(b)}, \eqref{sequences(c)} and
  \eqref{sequences(d)} are satisfied and additionally $f_n$ is locally $L_n'$-bilipschitz for some
  $L_n'< L_n$. 
  Let $\e_{n+1}'$ be so small that $20\e_{n+1}'$ is tubular for
  $f_n$, that $\bar B(f_n,\e_{n+1}')\subset T_{\l_n}$, that a tubular neighborhood of thickness $\e_{n+1}'$ can
  be realized by a locally $L_n$-bilipschitz function
  $g'_{n}\colon T_{\e_{n+1}'}\to T_{\e_n}$ and so that $2\e_{n+1}'<\min\{\l_n,2^{-n-2}\}$.  These are
  possible respectively by Theorem~\ref{thm:Tub}, by
  compactness, by Lemma~\ref{lemma:TubularBilipschitz} and the rest trivially. So let
  $g'_{n}$ be this function. Now $g_n'$ satisfies \eqref{sequences(a)} 
  and~\eqref{sequences(e)}, so everything except \eqref{sequences(f)}
  is satisfied. 
  By compactness and the fact that it is locally $L_n$-colipschitz,
  there exists $\e$ so that $g_n'$ is 
  $\e^{-1}$-colipschitz. Let $\e_{n+1}\le \e_{n+1}'$ be so small that $\e_{n+1}/\e_{n}^2<\e$.
  Then let $g_n=(g_n'\rest T_{\e_{n+1}})$. Now $f_n=g_n\rest S^1$ and all
  conditions are satisfied. Now define $\l_{n+1}<\e_{n+1}/2$, $L_{n+1}$ anything between
  $1$ and $2^{2^{-n-1}}$ and $f_{n+1}$ to be from $S^1$ to $T_{\e_{n+1}}$ satisfying
  \eqref{sequences(b)}, \eqref{sequences(c)} and \eqref{sequences(d)} and which is locally
  $L_{n+1}'$-colipschitz for some $L_{n+1}'< L_{n+1}$.
\end{proof}

\begin{figure}
  \centering
  \includegraphics[width=0.85\textwidth]{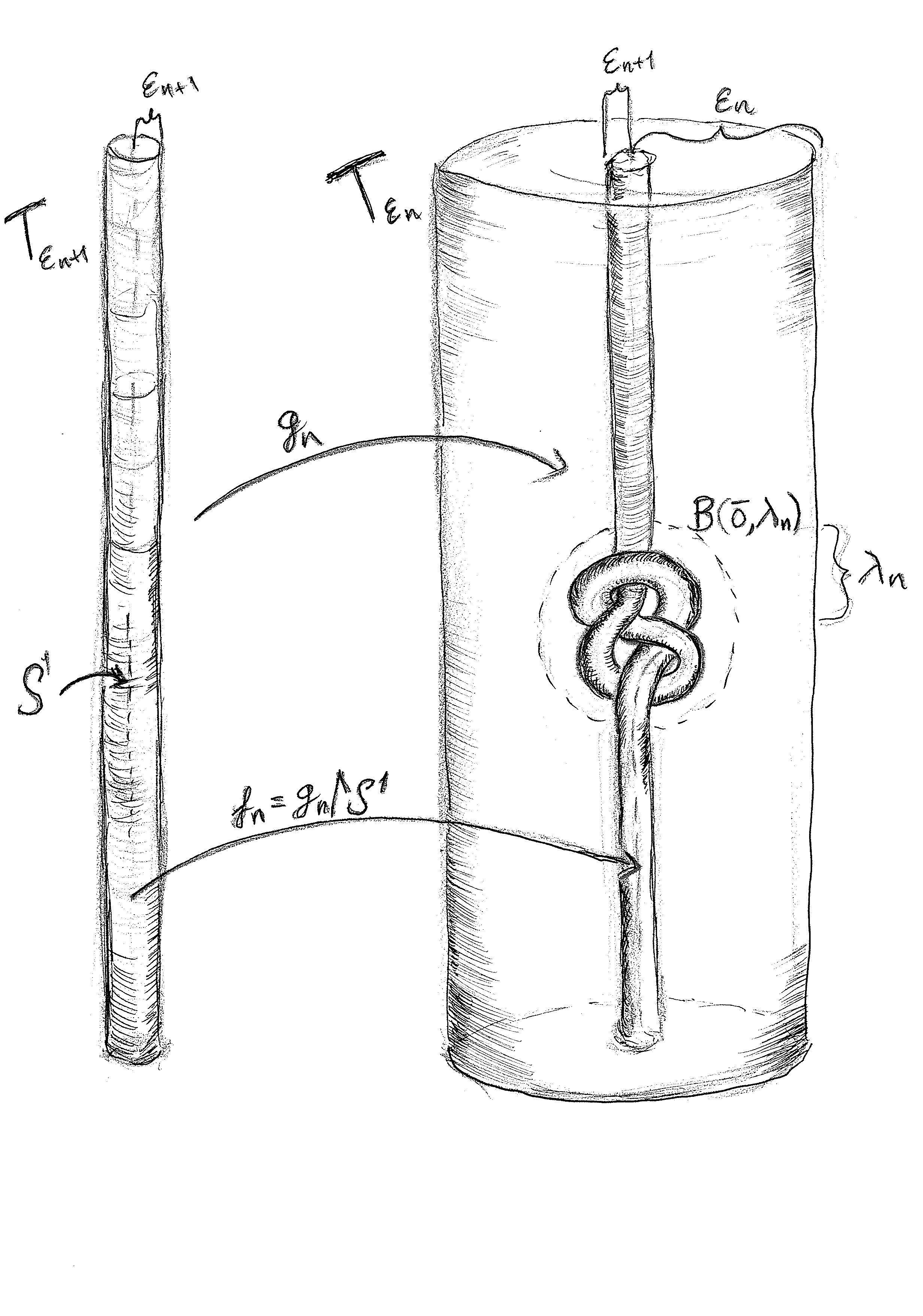}
  \caption{An illustration of Proposition~\ref{prop:SequencesOfEmbeddings}.}
  \label{fig:Embedding}
\end{figure}

Let $\theta^{s}\colon T_1\to T_1$ be the rotation by $s$: $\theta^{s}(b,t)=(b,t+s)$.
Note that $\theta^{s}\theta^{t}=\theta^{s+t}$.
Define $H_1^{s}$ to be the homeomorphism of $T_1$ defined by:
$H_1^s(b,t)=\theta^{r(b,s)}(b,t)$ where $r(b,s)=(1-2\max\{0,|b|-1/2\})s$
and let $H^s_\e=\b_{\e} H^s\b^{-1}_{\e}$ where $\b_\e\colon T_1\to T_{\e}$
is the homeomorphism $\b_\e(b,s)=(\e b,s)$.
It is clear that 
\begin{equation}
  \label{eq:IsTheta}
  H_\e^{s}\rest T_{\e/2}=\theta^{s} 
\end{equation}
and for all $x\in T_{\e/2}$ we have
\begin{equation}
  d(x,H_{\e}^s(x))=d(x,\theta^s(x))=|s|\label{eq:isometry}
\end{equation}
and in general $d(x,H_{\e}^{s}(x))\le |s|$.  
Also for all $x,y\in T_{\e}$ we have
\begin{equation}
  \label{eq:HomeoLip}
  d(H_\e^s(x),H_\e^{s}(y))\le d(x,y)+|s|.
\end{equation}

Let $\hat L_n=\prod_{i=0}^{n}L_n$. By \eqref{sequences(g)}, 
\begin{equation}
  \label{eq:hatLn}
  \hat L_n< 5.
\end{equation}

Let $\bar x$ be a sequence in $(S^1)^{\N}$. Define
$\hat g^{\bar x}_n\colon T_{\e_{n+1}}\to T_1$ by
\begin{equation}
  \hat g^{\bar x}_n=\comp_{k=0}^{n} \theta^{x_k} g_k\theta^{-x_k}.\label{eq:rotationform}
\end{equation}
It is easy to calculate that $\hat g^{\bar x}_n$ is locally $\hat L_n$-bilipschitz,
because $\theta^s$ is an isometry and by \eqref{sequences(e)}. By \eqref{eq:hatLn}
it is locally $5$-bilipschitz. In particular we have:

\begin{Prop}\label{prop:IsLip}
  $\hat g^{\bar x}_n$ is $5$-Lipschitz. 
\end{Prop}
\begin{proof}
  It follows from the fact that the standard metric on
  $\dom g^{\bar x}_n=T_{\e_{n+1}}$ coincides with its path metric
  and Proposition~\ref{prop:BilipPathMetric}.
\end{proof}

Equip each $\Im(\hat g_n)$ with
the path metric $d^{p(n)}$. From Proposition~\ref{prop:BilipPathMetric} we obtain:

\begin{Prop}\label{prop:Hbilipschitz} 
  $\hat g_n$ is $5$-bilipschitz as a function from $T_{\e_{n+1}}$ to
  $(\Im \hat g_n,d^{p(n)})$. \qed
\end{Prop}

Let $\hat f^{\bar x}_n=\hat g^{\bar x}_n\rest S^1$ and define
\begin{equation}\label{eq:defknot}
  K(\bar x)=\Cap_{n\in\N} \Im \hat g^{\bar x}_n
\end{equation}
A slight modification of
this is going to be 
the knot associated with $\bar x$ in the reduction. 
First let us show that $K(\bar x)$ is indeed a knot and is in fact the image
of $f^{\bar x}$ which is the limit of the sequence $(\hat f^{\bar x}_n)$.

\begin{Prop}\label{prop:DefOfKnot}
  $K(\bar x)$ is a knot, i.e. an injective continuous image of~$S^1$.
\end{Prop}
\begin{proof}
  For the time of this proof let us drop the ``$\bar x$'' from the upper case:
  $\hat g_n^{\bar x}=\hat g_n$, $\hat f_n^{\bar x}=\hat f_n$.
  \begin{eqnarray}
    d(\hat f_n(s),\hat f_{n+1}(s))&=& d(\hat g_n(s),\hat g_n \theta^{x_{n+1}}g_{n+1}\theta^{-x_{n+1}}(s))\nonumber\\
    &\le& 5 d(s,\theta^{x_{n+1}}f_{n+1}\theta^{-x_{n+1}}(s)),\label{eqnn1}\\
    &=& 5 d(\theta^{-x_{n+1}}s,f_{n+1}\theta^{-x_{n+1}}(s)),\label{eqnn2}\\
    &=& 5 d(t,f_{n+1}(t)),\nonumber\\
    &\le& 10\l_{n+1}\label{eqnn3}    
  \end{eqnarray}
  where $t=\theta^{-x_{n+1}}(s)$. \eqref{eqnn1} follows from
  Proposition~\ref{prop:IsLip},
  \eqref{eqnn2} follows from the fact that $\theta^s$ is an
  isometry and \eqref{eqnn3} follows from \eqref{sequences(b)} and \eqref{sequences(c)}. 
  From \eqref{sequences(g)} it follows in particular that $\sum_{k=n}^{\infty}\l_k\le \e_n$, so
  from the arbitrariness of $n$ above, the sequence $(\hat f_n)$ is
  Cauchy. The limit $f$ is contained in $\Im \hat g_n$ for every $n$,
  so $\Im f\subset K(\bar x)$ on one hand, and if a point is in
  $\Im \hat g_n$, then its distance from $\Im f$ is at most $\e_n$
  which tends to zero, so $K(\bar x)\subset \Im f$ on the other.
  
  It follows from \eqref{sequences(f)} that the limit, $f$, is injective:
  Let $s,t\in S^1$ and let $\e=d(s,t)$.  Let $n$ be so large that
  $20\e_{n-1}<\e$. From the above we know that $d(\hat f_n,f)\le 10\sum_{m>n}\l_n$ 
  which is at most $10\e_{n+1}$ by \eqref{sequences(g)}.
  Now
  \begin{eqnarray*}
    d(\hat f_{n}(s),\hat f_n(t))&\le& d(\hat f_n(s),f(s))+d(f(s),f(t))+d(\hat f_n(t),f(t))\\
    &\le& 20\e_{n+1} + d(f(s),f(t)).
  \end{eqnarray*}
  So we have
  $$d(f(s),f(t)) \ge d(\hat f_{n}(s),\hat f_n(t)) -20\e_{n+1}.$$
  But by \eqref{sequences(f)}, $\hat f_n$ is
  $M^{-1}$-colipschitz, where
  \begin{equation}
    M=\prod_{m=1}^{n}\frac{\e_m}{\e^2_{m-1}}\ge \e_n/\e_{n-1}.\label{eq:colpischitz}
  \end{equation}
  This is obtained by multiplying the colipschitz constants 
  of the factors of $\hat f_n$. Thus
  \begin{align*}
    d(f(s),f(t)) &\ge d(\hat f_{n}(s),\hat f_n(t)) -20\e_{n+1}\\
    &\ge \frac{\e_n}{\e_{n-1}}d(s,t)-20\e_{n+1} \\
    &= \frac{\e_n}{\e_{n-1}}\e-20\e_{n+1} \\
    &\ge \frac{\e_n}{\e_{n-1}}\cdot 20 \e_{n-1}-20\e_{n+1} \\
    &\ge 20(\e_n-\e_{n+1}) \\
    &> 0.      \qedhere
  \end{align*}
\end{proof}

Let $\bar x'$ be another sequence. Let us again  simplify the notation
$\hat g^{\bar x}_n=\hat g_n$ and $\hat g^{\bar x'}_{n}=\hat g'_n$.

For each $n\in\N$ let $y_{n+1}=(x'_{n+1}-x_{n+1})-(x'_{n}-x_{n})$.
Define homeomorphisms of $T_1$ as follows:
Let $H_0=H_{\e_0}^{x_0'-x_0}$ and 
\begin{equation}
  \label{eq:DefOfH_n}
  H_{n+1}(x)=
  \begin{cases}
    H_n\hat g_n H^{y_{n+1}}_{\e_{n+1}} \hat g_n^{-1}(x) &\text{ if } x\in\Im(\hat g_n) \\
    H_n & \text{ otherwise.}
  \end{cases}
\end{equation}
This is a homeomorphism for all $n$ which can be seen by induction. Suppose $H_n$ is a homeomorphism,
then $H_{n+1}$ is defined in two parts. Suppose $x$ is on the boundary: $x\in \partial \Im(\hat g_n)$.
But then $\hat g_n^{-1}x\in \partial T_{\e_{n+1}}$ and on this set $H^{y_{n+1}}_{\e_{n+1}}$ is the identity.
Thus $H_n\hat g_n H^{y_{n+1}}_{\e_{n+1}} \hat g_n^{-1}(x)=H_n(x)$.

Now if $x$ happens to be in 
$\Im(\hat g_{n+1})$, i.e. $x=\hat g_{n+1}(y)$ for some $y$,
then
\begin{eqnarray*}
  \hat g_n^{-1}(x)&=&\hat g_n^{-1}\hat g_{n+1}(y)\\
  &=&\hat g_n^{-1}\hat g_{n}\theta^{x_{n+1}}g_{n+1}\theta^{-x_{n+1}}(y)\\
  &=&\theta^{x_{n+1}}g_{n+1}\theta^{-x_{n+1}}(y)
\end{eqnarray*}
and so $\hat g^{-1}_n(x)$ is in $T_{\e_{n+1}/2}$ by \eqref{sequences(a)} and \eqref{sequences(g)}. Therefore by
\eqref{eq:IsTheta}, if $x\in \Im\hat g_{n+1}$, then
we have
\begin{equation}
  H_{n+1}(x)=H_n\hat g_n  \theta^{y_{n+1}} \hat g_n^{-1}(x)\label{eq:RotFormOfH_n}
\end{equation}

\begin{Prop}\label{prop:Hggpr}
  $H_n\hat g_n=\hat g_n'\theta^{x_n'-x_n}$.
\end{Prop}
\begin{proof}
  Induction: 
  \begin{eqnarray}
    H_0\hat g_0&=&H_{\e_0}^{x'_0-x_0}\theta^{x_0}g_0\theta^{-x_0}\nonumber \\
    &=&\theta^{x'_0-x_0}\theta^{x_0}g_0\theta^{-x_0}\label{steninaneq}\\
    &=&\theta^{x_0'}g_0\theta^{-x_0'}\theta^{x_0'-x_0}\nonumber\\
    &=&\hat g_0'\theta^{x_0'-x_0}.\nonumber
  \end{eqnarray}
  Step \eqref{steninaneq} follows from the fact that restricted to the range of $g_0$,
  $H_{\e_0}^{s}=\theta^s$.
  
  For $n+1$:
  \begin{align*}
    H_{n+1}\hat g_{n+1}
    &=H_{n}\hat g_n H^{y_{n+1}}_{\e_{n+1}}(\hat g_n)^{-1}\hat g_{n+1}\\
    &=H_{n}\hat g_n \theta^{y_{n+1}}(\hat g_n)^{-1}\hat g_{n+1}\\
    &=H_{n}\hat g_n \theta^{y_{n+1}}\theta^{x_{n+1}} g_{n+1}\theta^{-x_{n+1}}\\
    &=\hat g_n'\theta^{x_n'-x_n}\theta^{y_{n+1}}\theta^{x_{n+1}}g_{n+1}\theta^{-x_{n+1}}\\
    &=\hat g_n'\theta^{x_{n+1}'}g_{n+1}\theta^{-x'_{n+1}}\theta^{x'_{n+1}-x_{n+1}}\\
    &=\hat g_{n+1}'\theta^{x'_{n+1}-x_{n+1}}. \qedhere
  \end{align*}
\end{proof}
In particular, since $\theta^{s}$ is a bijection, $\Im(H_n \hat g_n)=\Im(\hat g_n')$, so
$(H_n\rest \Im g_n)$ is a homeomorphism from $\Im(g_n)$ to $\Im(g_n')$.
Since the domain of $\hat g_{n+1}$ is $T_{\e_{n+2}}$ and restricted to this domain
we have by \eqref{sequences(g)} and \eqref{eq:IsTheta}
$$\theta^{x'_{n+1}-x_{n+1}}\rest T_{\e_{n+2}}=H_{\e_{n+1}}^{x'_{n+1}-x_{n+1}}\rest T_{\e_{n+2}},$$
we can also write: 
\begin{equation}
  \label{eq:yetanotherform}
  H_n\hat g_n=\hat g_n' H_{\e_{n+1}}^{x_n'-x_n}.  
\end{equation}

Now, when $x\in \Im \hat g_{n+1}$
we can replace \eqref{eq:DefOfH_n} by the equivalent
$$H_{n+1}=\hat g'_{n}\theta^{x_{n+1}'-x_{n+1}}\hat g_{n}^{-1}(x),$$
because 
\begin{eqnarray}
  H_{n+1}&=&H_n\hat g_nH_{\e_{n+1}}^{y_{n+1}}\hat g_{n}^{-1}\nonumber\\
  &=&H_n\hat g_n\theta^{y_{n+1}}\hat g_{n}^{-1}\label{eqstep2}\\
  &=& \hat g'_n\theta^{x'_{n}-x_n}\theta^{y_{n+1}}\hat g_n^{-1}\nonumber\\
  &=& \hat g'_n\theta^{x'_{n+1}-x_{n+1}}\hat g_n^{-1}(x).\nonumber
\end{eqnarray}
Where \eqref{eqstep2} follows from \eqref{eq:RotFormOfH_n}.
On the other hand, a similar calculation using \eqref{eq:yetanotherform}
gives for $x\in \Im(\hat g_{n})$
$$H_{n+1}=\hat g'_{n}H_{\e_{n+1}}^{x_{n+1}'-x_{n+1}}\hat g_{n}^{-1}(x),$$
because
\begin{eqnarray}
  H_{n+1}&=&H_n\hat g_nH_{\e_{n+1}}^{y_{n+1}}\hat g_{n}^{-1}\nonumber\\
  &=& \hat g'_n H^{x'_{n}-x_n}_{\e_{n+1}} H_{\e_{n+1}}^{y_{n+1}}\hat g_n^{-1}\nonumber\\
  &=& \hat g'_n H_{\e_{n+1}}^{x'_{n+1}-x_{n+1}}\hat g_n^{-1}(x).\nonumber
\end{eqnarray}
So we have
\begin{eqnarray}
  H_{n+1}(x)=\hat g'_{n}\theta^{x_{n+1}'-x_{n+1}}\hat g_{n}^{-1}(x),&&\text{ if }x\in \Im (\hat g_{n+1})
  \label{eq:HatFormOfH1}\\
  H_{n+1}(x)=\hat g'_{n}H_{\e_{n+1}}^{x_{n+1}'-x_{n+1}}\hat g_{n}^{-1}(x),&&\text{ if }x\in \Im (\hat g_{n})
  \label{eq:HatFormOfH2}
\end{eqnarray}

\begin{Prop}\label{prop:Hbilipschitz2} 
  Suppose $x,z\in \Im\hat g_{n-1}$. Then 
  $$d(H_n(x),H_n(z))\le 25d^{p(n-1)}(x,z)+5|x'_n-x_n|.$$
\end{Prop}
\begin{proof}
  Using Propositions \ref{prop:IsLip} and \ref{prop:Hbilipschitz}, and
  \eqref{eq:HatFormOfH2} and \eqref{eq:HomeoLip}
  we can calculate:
  \begin{align*}
    d(H_n(x),H_n(z))&=d(\hat g'_{n-1}H^{x_n'-x_n}_{\e_n}\hat g^{-1}_{n-1}(x),\hat g'_{n-1}H^{x_n'-x_n}_{\e_n}\hat g^{-1}_{n-1}(z))\\
    &\le 5 d(H^{x_n'-x_n}_{\e_n}\hat g^{-1}_{n-1}(x),H^{x_n'-x_n}_{\e_n}\hat g^{-1}_{n-1}(z))\\
    &\le 5 d(\hat g^{-1}_{n-1}(x),\hat g^{-1}_{n-1}(z))+5|x_n'-x_m|\\
    &\le 25 d^{p(n-1)}(x,z)+5|x_n'-x_m|. \qedhere
  \end{align*}
\end{proof}

\begin{Prop}\label{prop:isoLip}
  Suppose $Y\subset X$ are metric spaces, $\theta\colon X\to X$ is an isometric homeomorphism, and
  $g\colon Y\to X$ is an embedding. 
  Additionally assume that $\l=\sup_{y\in Y}d(y,gy)<\infty$ and $x\in \Im g$, $y\in Y$.
  Then
  $$d(x,\theta g y)\le d(g^{-1}x,\theta y)+2\l.$$
\end{Prop}
\begin{proof}
  \begin{align*}
    d(x,\theta g y)&= d(\theta^{-1}x, gy) \\
    &\le  d(\theta^{-1}x,y)+d(y,gy) \\
    & =  d(x,\theta y) + d(y,gy) \\
    & \le  d(g^{-1}x,\theta y) +d(g^{-1}x,x)+ d(y,gy) \\
    &\le  d(g^{-1}x,\theta y)+2\l. \qedhere
  \end{align*}
\end{proof}

\begin{Prop}\label{supertechnicalclaim}
  Assume $k<n$ and suppose $\theta_i$ and $\theta_i'$ are 
  isometries of $T_1$ of the form $\theta^{s}$ for $i\in \{k,k+1,\dots, n\}$.
  and $t\in \Im g_n$. Then
  $$d\big(t,\comp_{i=k}^{n}g_i\theta_i\circ \comp_{i=n}^{k}\theta'_{i}g_{i}^{-1}(t)\big)\le d\big(t,\comp_{i=k}^{n}\theta_i\theta_i'(t)\big)+4\sum_{i=k}^{n}\l_i.$$
\end{Prop}
\begin{proof}
  Applying Proposition~\ref{prop:isoLip} $2(n-k+1)$ times in a row one obtains the result.
  The first of these applications is the special case where $\theta=\theta^{0}=\id$:
  \begin{equation*}
    d(x,g y)\le d(g^{-1}x, y)+2\l.\qedhere
  \end{equation*}
\end{proof}

\begin{Prop}\label{prop:anotherineq}
  Fix $n$ and $k$ and let $x\in \Im\hat g_{n+k}$. Then
  $$d(H_{n}(x),H_{n+k}(x))\le 10\sup_{m\ge n}|x_m'-x_m|+20\e_n.$$
\end{Prop}
\begin{proof}
  Since $x\in \Im\hat g_{n+k}$, we can write by \eqref{eq:HatFormOfH1}:
  \begin{eqnarray*}
    H_{n}(x)&=&\hat g'_{n-1}\theta^{x'_{n}}\theta^{-x_n}\hat g_{n-1}^{-1}(x)\\    
    \text{and }H_{n+k}(x)&=&\hat g'_{n+k-1}\theta^{x'_{n+k}}\theta^{-x_{n+k}}\hat g_{n+k-1}^{-1}(x).
  \end{eqnarray*}
  Applying \eqref{eq:rotationform} we can further rewrite the
  expression for $H_{n+k}$:
  $$H_{n+k}(x)=\hat g'_{n-1}\comp_{i=n}^{n+k-1}\theta^{x_i'}g_i\theta^{-x_i'}
  \circ \theta^{x'_{n+k}}\theta^{-x_{n+k}}\hat g_{n+k-1}^{-1}(x).$$
  Adopting notations
  \begin{eqnarray*}
    a&=&\theta^{x'_{n}}\theta^{-x_n}\hat g_{n-1}^{-1}(x)\text{ and }\\
    b&=&\comp_{i=n}^{n+k-1}\theta^{x_i'}g_i\theta^{-x_i'} \circ \theta^{x'_{n+k}}\theta^{-x_{n+k}}\hat g_{n+k-1}^{-1}(x)
  \end{eqnarray*}
  and using Proposition~\ref{prop:IsLip} we get:
  \begin{equation}
    \label{eq:blahh2}
    d(H_n(x),H_{n+k}(x))\le d(\hat g'_{n-1}(a),\hat g'_{n-1}(b))\le 5 d(a,b).
  \end{equation}
  Now we can further rewrite $b$ by expanding the term $\hat g^{-1}_{n+k-1}$
  using \eqref{eq:rotationform} as
  \begin{eqnarray*}
    \hat g^{-1}_{n+k-1}&=&\comp_{i=n+k-1}^{0}\theta^{x_i}g^{-1}_i\theta^{-x_i}\\
    &=&\comp_{i=n+k-1}^{n}\theta^{x_i}g^{-1}_i\theta^{-x_i} \circ \comp_{i=n-1}^{0}\theta^{x_i}g^{-1}_i\theta^{-x_i}\\
    &=&\comp_{i=n+k-1}^{n}\theta^{x_i}g^{-1}_i\theta^{-x_i}\circ \hat g_{n-1}^{-1}\\
    &=&\comp_{i=n+k-1}^{n+1}\theta^{x_i}g^{-1}_i\theta^{-x_i}\circ \theta^{x_n}g^{-1}_n\theta^{-x_n}\hat g_{n-1}^{-1}\\
  \end{eqnarray*}
  so
  $$b=\comp_{i=n}^{n+k-1}\theta^{x_i'}g_i\theta^{-x_i'} \circ \theta^{x'_{n+k}}\theta^{-x_{n+k}}\circ\comp_{i=n+k-1}^{n+1}\theta^{x_i}g^{-1}_i\theta^{-x_i}\circ \theta^{x_n}g^{-1}_n\circ \theta^{-x_n}\hat g_{n-1}^{-1}(x).$$
  Now, using the substitution $t=\theta^{-x_n}\hat g_{n-1}^{-1}(x)$
  rewrite both $a$ and $b$ again:
  \begin{eqnarray*}
    a&=&\theta^{x'_{n}}(t),\\
    b&=&\comp_{i=n}^{n+k-1}\theta^{x_i'}g_i\theta^{-x_i} \circ \theta^{x'_{n+k}}\theta^{-x_{n+k}}\comp_{i=n+k-1}^{n+1}\theta^{x_i}g^{-1}_i\theta^{-x_i}\circ \theta^{x_n}g^{-1}_n(t)
  \end{eqnarray*}
  Because $\theta^{-x'_n}$ is an isometry, we have
  \begin{equation}
    d(a,b)=d(\theta^{-x'_n}a,\theta^{-x'_n}b)=d(t,\theta^{-x'_n}b).\label{eq:eqblah1}
  \end{equation}
  So let us define $c=\theta^{-x'_n}b$.  Now
  \begin{eqnarray*}
    c&=&\theta^{-x'_n}\comp_{i=n}^{n+k-1}\theta^{x_i'}g_i\theta^{-x_i'} \circ 
    \theta^{x'_{n+k}}\theta^{-x_{n+k}} \comp_{i=n+k-1}^{n+1}\theta^{x_i}g^{-1}_i\theta^{-x_i}
    \circ \theta^{x_n} g^{-1}_n(t)\\
    &=&\theta^{-x'_n}\theta^{x'_n}g_n\theta^{-x'_n}\comp_{i=n+1}^{n+k-1}\theta^{x_i'}g_i\theta^{-x_i'} 
    \circ \theta^{x'_{n+k}}\theta^{-x_{n+k}}\comp_{i=n+k-1}^{n+1}\theta^{x_i}g^{-1}_i\theta^{-x_i}
    \circ \theta^{x_n}g^{-1}_n(t)\\
    &=&g_n\theta^{-x_n}\comp_{i=n+1}^{n+k-1}\theta^{x_i'}g_i\theta^{-x_i'} \circ \theta^{x'_{n+k}}\theta^{-x_{n+k}}\comp_{i=n+k-1}^{n+1}\theta^{x_i}g^{-1}_i\theta^{-x_i}\circ \theta^{x_n}g^{-1}_n(t)
  \end{eqnarray*}
  Now re-ordering the terms, we get:
  $$c=\comp_{i=n}^{n+k-1}g_i\theta^{x_{i+1}'-x_i'} \circ \comp_{i=n+k-1}^{n}\theta^{-(x_{i+1}-x_i)}g^{-1}_i(t).$$
  At this point we see that $c$ is of the form that we can apply
  Proposition~\ref{supertechnicalclaim} to $d(t,c)$. Thus:
  \begin{eqnarray*}
    d(t,c)&=&d\left(t,\comp_{i=n}^{n+k-1}g_i\theta^{x_{i+1}'-x_i'} \circ \comp_{i=n+k-1}^{n}\theta^{-(x_{i+1}-x_i)}g^{-1}_i(t)\right)\\
    &\le& d\left(t,\comp_{i=n+1}^{n+k}\theta^{(x_{i+1}'-x'_{i})-(x_{i+1}-x_i)} (t)\right) + 4\sum_{i=n}^{n+k}\l_i\\
    &=& d\left(t,\comp_{i=n+1}^{n+k}\theta^{(x_{i+1}'-x_{i+1})-(x'_{i}-x_i)} (t)\right) + 4\sum_{i=n}^{n+k}\l_i\\
    &=& d\left(t,\comp_{i=n+1}^{n+k}\theta^{y_{i+1}} (t)\right) + 4\sum_{i=n}^{n+k}\l_i.
  \end{eqnarray*}
  Now, $\comp_{i=n+1}^{n+k}\theta^{y_{i+1}}=\theta^{s}$ where $s=\sum_{i=n+1}^{n+k}y_{i+1}$. But an
  easy calculation shows that in fact $s=(x'_{n+k+1}-x_{n+k+1})-(x'_{n+1}-x_{n+1})$, so we obtain
  $$d(t,c)\le d\left(t,\theta^{(x'_{n+k+1}-x_{n+k+1})-(x'_{n+1}-x_{n+1})} (t)\right) + 4\sum_{i=n}^{n+k}\l_i,$$
  so applying  \eqref{eq:isometry} we get
  \begin{eqnarray}
    d(t,c)&\le &(x'_{n+k+1}-x_{n+k+1})-(x'_{n+1}-x_{n+1}) + 4\sum_{i=n}^{n+k}\l_i\nonumber\\
    &\le &|x'_{n+k+1}-x_{n+k+1}|+|x'_{n+1}-x_{n+1}| + 4\sum_{i=n}^{n+k}\l_i \nonumber\\
    &\le &2\sup_{m\ge n}|x'_{m}-x_{m}| + 4\sum_{i=n}^{n+k}\l_i \label{eq:blah3}
  \end{eqnarray}

  Now combining \eqref{eq:blahh2}, \eqref{eq:eqblah1}, the fact that $c=\theta^{-x'_n}b$, \eqref{eq:blah3}
  and \eqref{sequences(g)} we get:
  \begin{align*}
    d(H_n(x),H_{n+k}(x))&\le2\cdot 5\sup_{m\ge n}|x'_{m}-x_{m}| + 4\cdot 5\sum_{i=n}^{n+k}\l_i\\
    &\le 10\sup_{m\ge n}|x'_{m}-x_{m}| + 20\sum_{i=n}^{\infty}\l_i\\
    &\le 10\sup_{m\ge n}|x'_{m}-x_{m}| + 20\e_n.\qedhere
  \end{align*}
\end{proof}

\begin{Prop}\label{prop:anotherineq2}
  Fix $n$ and $k\ge 1$ and let $x\in \Im\hat g_{n+k-1}$. Then
  $$d(H_{n}(x),H_{n+k}(x))\le 15\sup_{m\ge n}|x_m'-x_m|+50\e_n.$$
\end{Prop}
\begin{proof}
  Let $y\in T_{\e_{n+k}}$ be such that $\hat g_{n+k-1}(y)=x$ 
  and let
  $y^*\in \Im (\theta^{x_{n+k}}g_{n+k}\theta^{-x_{n+k}})$ with $d(y,y^*)\le \e_{n+k}$.
  Denote $x^*= \hat g_{n+k-1}(y^*)$. 
  Now $x^*\in \Im (\hat g_{n+k-1}\circ \theta^{x_{n+k}}g_{n+k}\theta^{-x_{n+k}})=\Im(\hat g_{n+k})$.
  By Proposition \ref{prop:IsLip} we have that
  $d(x^*,x)\le 5\e_{n+k}$. Let $z^*=\hat g_{n-1}^{-1}(x^*)$ and $z=\hat g_{n-1}^{-1}(x)$. Now
  \begin{eqnarray*}
    z^*&=&\hat g_{n-1}^{-1}\hat g_{n+k-1}(y^*) \\
    &=&\comp_{k=n-1}^{0} \theta^{x_k} g^{-1}_k\theta^{-x_k}\circ\comp_{k=0}^{n+k-1}\theta^{x_k} g^{-1}_k\theta^{-x_k}(y^*)\\
    &=&\comp_{k=n}^{n+k-1}\theta^{x_k} g_k\theta^{-x_k}(y^*).
  \end{eqnarray*}
  Similarly
  $$z=\comp_{k=n}^{n+k-1}\theta^{x_k} g^{-1}_k\theta^{-x_k}(y).$$
  So we have $d(z,z^*)=d(h(y),h(y^*))$ where 
  $$h=\comp_{k=n}^{n+k-1}\theta^{x_k} g_k\theta^{-x_k}.$$
  By the same argument as in the proof of Proposition~\ref{prop:IsLip}
  we have that $h$ is $5$-Lipschitz, so
  $$d(z,z^*)\le 5d(y,y^*)\le 5\e_{n+k}.$$
  Additionally from \eqref{eq:HomeoLip}  
  we have
  $$d(H_{\e_{n+k}}^{x_{n+k}'-x_{n+k}}(y),H_{\e_{n+k}}^{x_{n+k}'-x_{n+k}}(y^*))\le d(y,y^*)+|x_{n+k}'-x_{n+k}|.$$

  Now from triangle inequality we get:
  $$d(H_{n}(x),H_{n+k}(x))\le  d(H_{n}(x),H_{n}(x^*))+d(H_{n}(x^*),H_{n+k}(x^*))+d(H_{n+k}(x^*),H_{n+k}(x))$$
  Let us consider the three terms separately.
  The first term  by
  \eqref{eq:HatFormOfH1} and Proposition~\ref{prop:IsLip}:
  \begin{eqnarray*}
    d(H_{n}(x),H_{n}(x^*)) &\le& d(\hat g'_{n-1}\theta^{x_n'-x_n}\hat g_{n-1}^{-1}(x),\hat g'_{n-1}\theta^{x_n'-x_n}\hat g_{n-1}^{-1}(x^*))\\
    &=& d(\hat g'_{n-1}\theta^{x_n'-x_n} (z),\hat g'_{n-1}\theta^{x_n'-x_n}(z^*))\\
    &\le & 5d(\theta^{x_n'-x_n} (z),\theta^{x_n'-x_n}(z^*))\\
    &=& 5d(z,z^*) \\
    &\le& 25\e_{n+k}\\
    &\le& 25\e_{n}.
  \end{eqnarray*}
  The second term using Proposition~\ref{prop:anotherineq}:
  $$d(H_{n}(x^*),H_{n+k}(x^*))\le 10\sup_{m\ge n}|x_m'-x_m|+20\e_n.$$
  And the third term using \eqref{eq:HatFormOfH2}, \eqref{eq:HomeoLip}, Proposition~\ref{prop:IsLip} and the facts
  that $y=\hat g^{-1}_{n+k-1}(x)$, $y^*=\hat g^{-1}_{n+k-1}(x^*)$ and $d(y,y^*)<\e_{n+k}$:
  \begin{eqnarray*}
    d(H_{n+k}(x^*),H_{n+k}(x))&\le&     
    d(\hat g'_{n+k-1}H_{\e_{n+k}}^{x_{n+k}'-x_{n+k}}\hat g_{n+k-1}^{-1}(x),\hat g'_{n+k-1}H_{\e_{n+k}}^{x_{n+k}'-x_{n+k}}\hat g_{n+k-1}^{-1}(x^*))\\
    &=& d(\hat g'_{n+k-1}H_{\e_{n+k}}^{x_{n+k}'-x_{n+k}}(y),\hat g'_{n+k-1}H_{\e_{n+k}}^{x_{n+k}'-x_{n+k}}(y^*))\\
    &\le & 5d(H_{\e_{n+k}}^{x_{n+k}'-x_{n+k}}(y),H_{\e_{n+k}}^{x_{n+k}'-x_{n+k}}(y^*))\\
    &\le& 5d(y,y^*)+5|x'_{n+k}-x_{n+k}|\\
    &\le&5\e_{n+k}+5|x'_{n+k}-x_{n+k}|\\
    &\le&5\e_{n}+5\sup_{m>n}|x'_{m}-x_{m}|.
  \end{eqnarray*}
  Combining these we get:
  \begin{equation*}
    d(H_{n}(x),H_{n+k}(x))\le 15\sup_{m\ge n}|x_m'-x_m|+ 50\e_n.\qedhere
  \end{equation*}
\end{proof}

By the definition of $H_{n}$, for all $x\notin \Cap_{n\in\N}\Im (\hat g_{n})$
there is $m$ such that for all $n>m$ we have $H_m(x)=H_n(x)$. In fact
if $n$ is such that $x\notin \Im \hat g_n$, then $H_m(x)=H_n(x)$ for all $m>n.$
Thus the point-wise limit
$H_{\infty}=\lim_{n\to \infty}H_n$ is well defined on the complement of $K(\bar x)$
for which we have
\begin{equation}
  \label{eq:term}
  x\notin \Im \hat g_n \Rightarrow H_\infty(x)=H_n(x)
\end{equation}
$H_\infty$ is a homeomorphism taking the complement
of $K(\bar x)$ to the complement of~$K(\bar x')$.

Suppose $x\in \Im \hat g_{n-1}\setminus \Im \hat g_{n}$ and 
$z\in \Im\hat g_{n+k-1}\setminus \Im \hat g_{n+k}$, $k\ge 1$. Now from \eqref{eq:term} and
Propositions~\ref{prop:Hbilipschitz2} and~\ref{prop:anotherineq2} we have:
\begin{eqnarray*}
  d(H_{\infty}(x),H_\infty(z))&=&d(H_{n}(x),H_{n+k}(z))\\
  &\le& d(H_{n}(x),H_n(z))+d(H_n(z),H_{n+k}(z))\\
  &\le& 25d^{p(n-1)}(x,z)+5|x'_n-x_n| + 15\sup_{m\ge n}|x'_{m}-x_{m}| + 50\e_n.\\
  &\le& 25d^{p(n-1)}(x,z)+20\sup_{m\ge n}|x'_{m}-x_{m}| + 50\e_n.
\end{eqnarray*}

The latter is independent of $k$, so (changing $n-1$ to $n$) we have for all 
$x\in \Im \hat g_{n}\setminus \Im \hat g_{n+1}$ and $z\in\Im \hat g_{n+1}$ that:
\begin{equation}
  d(H_{\infty}(x),H_\infty(z))\le 100(d^{p(n)}(x,z) + \sup_{m> n}|x'_{m}-x_{m}| + \e_{n+1}). \label{eq:ImportantEq}
\end{equation}

\begin{Prop}\label{prop:Cauchy}
  If $(z_n)$ is a Cauchy sequence with $z_n\in \Im \hat g_{n}$ then for every $\e$ there is $n$ such that for
  all $m>n$ we have $d^{p(n)}(z_{n},z_{m})<\e$.
\end{Prop}
\begin{proof}
  If not, pick an $\e$ and a subsequence $(z_{n_k})$ so that for all
  $k$ we have $d^{p(n(k))}(z_{n(k)},z_{n(k+1)})\ge \e$ where $n_k=n(k)$.  
  Fix $k$ so large
  that $5\e_{n(k)-1}<\e$. Now for all $k'$ we have
  $d^{p(n(k'))}(z_{n(k')},z_{n(k'+1)})\ge \e \ge 5\e_{n(k)-1}$.  Now we can
  pick $k'>k$ to be so big that $d(z_{n(k')},z_{n(k'+1)})<\e_{n(k)}$. By
  Proposition~\ref{prop:Hbilipschitz}, we have
  $$d^{p(n(k'))}(z_{n(k')},z_{n(k'+1)})\le 5d(\hat g^{-1}_{n(k)}(z_{n(k')}),\hat g^{-1}_{n(k)}(z_{n(k'+1)})).$$ 
  So we have
  \begin{eqnarray*}
    5\e_{n(k)-1}&\le& d^{p(n(k'))}(z_{n(k')},z_{n(k'+1)})\\
    &\le& 5 d(\hat g^{-1}_{n(k)}(z_{n(k')}),\hat g^{-1}_{n(k)}z_{n(k'+1)})\\
    \iff\e_{n(k)-1}&\le&  d(\hat g^{-1}_{n(k)}z_{n(k')},\hat g^{-1}_{n(k)}z_{n(k'+1)})
  \end{eqnarray*}
  And because $\hat g_{n(k)}$ is $\e_{n(k)-1}/\e_{n(k)}$-colipschitz
  (cf.~\eqref{eq:colpischitz}),
  $$d(z_{n(k)},z_{n(k')})\ge \e_{n(k)}/\e_{n(k)-1}d(\hat g^{-1}_{n(k)}z_{n(k')},\hat g^{-1}_{n(k)}z_{n(k'+1)}) \ge \e_{n(k)}$$
  which is a contradiction with the choice of~$k'$.
\end{proof}

\begin{Prop}\label{prop:EquivalentKnots}
  If $\lim_{n\to\infty}(x'_{n}-x_n)=0$, then the knots $K(\bar x)$ and
  $K(\bar x')$ are equivalent.
\end{Prop}
\begin{proof}
  It remains to show that $H_\infty$ which is a homeomorphism from the complement of
  $K(\bar x)$ to the complement of $K(\bar x')$ extends to the knots themselves.
  It will follow if we show that a sequence $(z_n)$ in the complement of $K(\bar x)$
  is Cauchy if and only if $(H_\infty(z_n))$ is Cauchy in the complement of $K(\bar x')$.
  The ``if'' part will follow by symmetry once we prove the ``only if'' part.
  
  So suppose $(z_n)$ is Cauchy in the complement of $K(\bar x)$ and
  fix $\e$. For each $n$ let $k(n)$ be the largest number such that $z_n\in \Im\hat g_{k(n)}$.
  Without loss of generality assume that $k(n)$ is strictly increasing in~$n$.
  Choose $n$ so large that for all $m>n$,
  $100d^{p(n)}(z_{n},z_m)<\e/3$ which exists by
  Proposition~\ref{prop:Cauchy}, so that 
  $100\sup_{m\ge n+1}|x'_m-x_m|<\e/3$ 
  which is possible by the convergence, and so
  that $100\e_{n+1}<\e/3$ which is possible by
  \eqref{sequences(g)}.  
  Since $z_n\in \Im \hat g_{k(n)}\setminus \Im\hat g_{k(n)+1}$ and $z_m\in \hat g_{k(n)+1}$ for $m>n$,
  by inequality \eqref{eq:ImportantEq} we have for all $m>n$
  \begin{eqnarray*}
    d(H_{\infty}(z_{n}),H_\infty(z_m))&\le& 100(d^{p(k(n))}(z_n,z_m) + \sup_{m\ge k(n)+1}|x_m'-x_m| + \e_{k(n)+1})\\
    &\le&100(d^{p(n)}(z_n,z_m) + \sup_{m\ge n+1}|x_m'-x_m| + \e_{n+1})\\
    &\le& \frac{\e}{3} + \frac{\e}{3} +\frac{\e}{3}\\
    &=& \e.
  \end{eqnarray*}
  The second inequality follows simply from the fact that $k(n)\ge n$.
\end{proof}

Recall the definition of $f^{\bar x}$, just prior to Proposition~\ref{prop:DefOfKnot}.

\begin{Prop}\label{prop:hasasac}
  Suppose $(x_{n(k)})_k$ is a convergent subsequence of $(x_n)$
  which converges to some $a$. Let $\e>0$ be arbitrarily small and
  let $\g$ be the connected component of $f^{\bar x}[S^1]\cap B(f^{\bar x}(a),\e)$
  with $f^{\bar x}(a)\in \g$. Then there exists
  $k$ such that for all $k'>k$, 
  $(B(f^{\bar x}(a),\e),\g)$ contains $K_{n(k')}$ strongly as a component.
  (Recall Definition~\ref{def:asacomponent}.)
\end{Prop}
\begin{proof}
  For simplicity of notation denote $n_k=n(k)$.
  Let $k$ be so big that $5(d(x_{n(k)},a)+\l_{n(k)})<\e/2$.
  Let $k'> k$. Now 
  $$\theta^{x_{n(k')}}f_{n(k')}^{\bar x}\theta^{-x_{n(k')}}$$
  has $K_{n(k')}$ (strongly) as a component in $B(x_{n(k')},\l_{n(k')})$ by~\eqref{sequences(c)}
  and so
  $$\hat f^{\bar x}_{n(k')}=\hat g^{\bar x}_{n-1}\theta^{x_{n(k')}}f_{n(k')}^{\bar x}\theta^{-x_{n(k')}}$$
  contains $K_{n(k')}$ strongly as a component inside 
  $$\hat g_{n-1}^{\bar x}[B(x_{n(k')},\l_{n(k')})]$$
  and this is witnessed by $20\e_{n+1}$ by \eqref{sequences(h)} and Lemmas~\ref{lemma:asacomponent} 
  and~\ref{lemma:invarianceofcomp}.
  On the other hand $d(f^{\bar x},\hat f^{\bar x}_{n(k')})<10\e_{n+1}\le (20\e_{n+1})/2$
  which by Lemma~\ref{lemma:HerComp} means that $f^{\bar x}$ has $K_{n(k')}$ as a component in
  $$\hat g_{n-1}^{\bar x}[B(x_{n(k')},\l_{n(k')})]$$ as well.
  On the other hand, by the choice of $k$, 
  \begin{equation*}
    \hat g_{n-1}^{\bar x}[B(x_{n(k')},\l_{n(k')})]\subset B(f^{\bar x}(a),\e).\qedhere
  \end{equation*}
\end{proof}

\begin{Prop}\label{prop:hasnottasac}
  Suppose $a,a'\in S^1$, $n\in\N$ and $B\subset S^3$ are such that $B$ is
  a set homeomorphic to
  a closed ball with $f^{\bar x}(a)$ in the interior
  and $f^{\bar x}(a')$ in the exterior.
  Suppose $(x_{n(k)})$ is a subsequence converging to $a'$.
  Let $\g$ be the connected component
  of $f^{\bar x}\cap B$ which contains $f^{\bar x}(a)$. Then
  there exists $k$ such that for all $k'>k$, 
  $(B,\g)$ does not contain $K_{n(k')}$ strongly as a component.
\end{Prop}
\begin{proof}
  Since $B$ is closed, there is $B'$, also homeomorphic to a closed ball
  which contains $f^{\bar x}(x_n)$ in the interior and disjoint from $B$. 
  Let $\g'$ be the connected component of $f^{\bar x}\cap B'$ which 
  contains $f^{\bar x}(a')$. Then
  Let $k$ be as given
  by Proposition~\ref{prop:hasasac}, i.e. for all $k'>k$, $(B',\g')$ contains
  $K_{n(k')}$ strongly as a component. If now also $(B,\g)$ contains
  $K_{n(k')}$ strongly as a component, then it is easy to see from the definition of 
  strong components that $f^{\bar x}$ contains $K_{n(k')}\# K_{n(k')}$ strongly as a
  component. But this is impossible because 
  $d(f^{\bar x},\hat f^{\bar x}_m)\stackrel{m\to \infty}{\longrightarrow}0$
  and $\hat f^{\bar x}_m$ does not contain $K_{n(k')}\# K_{n(k')}$ strongly as a
  component for any~$m$.
\end{proof}

\begin{Prop}\label{prop:NonEquivalentKnots}
  Suppose $(x_n')$ and $(x_n)$ are such that there exists $A\subset \N$ so that
  for all $n\in A$ $x'_n=x_n$ and $\{x'_n\mid n\in A\}=\{x_n\mid n\in A\}$ 
  is dense.
  Then, if $\lim_{n\to\infty}(x_n'-x_n)$ does not exist, then
  the knots $K(\bar x)$ and $K(\bar x')$ are not equivalent.
\end{Prop}
\begin{proof}
  The set $\{(x'_n-x_n)\mid n\notin A\}$ must have at least one cluster point
  by the compactness of $S^1$ and at least one cluster point that is not equal
  to $0$, because otherwise $(x'_n-x_n)$ would converge to zero.
  Let $(n(k))_k$  be an increasing sequence of natural numbers so that
  $z=\lim_{k\to\infty}(x'_{n(k)}-x_{n(k)})$ exists and $z\ne 0$.
  Further, take a subsequence of this $(n(k(j)))_j$ such that
  $(x_{n(k(j))})_j$ and $(x'_{n(k(j))})_j$ converge to points $a, a'\in S^1$
  respectively. These points must be distinct, in fact $a'-a=z$. 
  We will now arrive at a contradiction from the assumption that
  there is a homeomorphism $h\colon S^3\to S^3$ with
  $h[K(\bar x)]=K(\bar x')$. 
  There are two cases:
  \begin{description}
  \item[Case 1] $h(f^{\bar x}(a))=f^{\bar x'}(a')$.
    Let $\e$ be so small that 
    $$B(f^{\bar x}(a),\e)\cap B(f^{\bar x}(a'),\e)=\es$$
    and
    $$B(f^{\bar x'}(a),\e)\cap h[B(f^{\bar x}(a),\e)]=\es.$$
    The last is possible because 
    $h(f^{\bar x}(a))=f^{\bar x'}(a')\ne f^{\bar x'}(a).$
    Since $\{x_m\mid m\in A\}$ is dense, 
    pick a subsequence $(x_{m(k)})$ 
    converging to $a$.
    Let $\g$ and $\g'$ be the connected components of 
    $f^{\bar x}\cap B(f^{\bar x}(a),\e)$
    and 
    $f^{\bar x'}\cap h[B(f^{\bar x}(a),\e)]$
    containing $f^{\bar x}(a)$ and $f^{\bar x'}(a')$
    respectively. Note that in fact $h[\g]=\g'$.
    By Proposition~\ref{prop:hasasac} there is $k\in A$ such that 
    $(B(f^{\bar x}(a),\e),\g)$
    has $K_{m(k_1)}$ strongly as a component for all $k_1>k$
    and by Proposition~\ref{prop:hasnottasac} there is $k'$
    such that
    $(h[B(f^{\bar x}(a),\e)],\g')$ does not have $K_{m(k_2)}$
    strongly as a component for any $k_2>k'$. By taking $k_3>\max\{k,k'\}$
    this contradicts Lemma~\ref{lemma:invarianceofcomp}.
  \item[Case 2] $h(f^{\bar x}(a))\ne f^{\bar x'}(a')$. 
    This is similar to Case 1, but it is also presented here for the sake of completeness.
    Let $\e$ be so small that 
    $$B(f^{\bar x}(a),\e)\cap B(f^{\bar x}(a'),\e)=\es$$
    and
    $$B(f^{\bar x'}(a'),\e)\cap h[B(f^{\bar x}(a),\e)]=\es.$$
    Let $\g$ and $\g'$ be the connected components of 
    $f^{\bar x}\cap B(f^{\bar x}(a),\e)$
    and 
    $f^{\bar x'}\cap h[B(f^{\bar x}(a),\e)]$
    containing $f^{\bar x}(a)$ and $f^{\bar x'}(a')$
    respectively. 
    By the above, $(x_{n(k(j))})$ converges to $a$, so
    by Proposition~\ref{prop:hasasac} there exists $j_0$ such that for all $j>j_0$,
    $(B(f^{\bar x}(a),\e),\g)$ contains $K_{n(k(j))}$ strongly as a component.
    On the other hand $(x'_{n(k(j))})$ converges to $a'$ which means by
    Proposition~\ref{prop:hasnottasac} that there is $j_1$ such that for all $j>j_1$
    $(h[B(f^{\bar x}(a),\e)],\g')$ does \emph{not} contain  $K_{n(k(j))}$ 
    strongly as a component. This is again a contradiction with
    Lemma~\ref{lemma:invarianceofcomp}. \qedhere
  \end{description}
\end{proof}

Now we are ready to prove the main theorem.

\begin{proof}[Proof of Theorem \ref{thm:Main}]\label{page:proofofmain}
  Fix a dense $(z_n)\subset S^1$. Let $(x_n)\in (S^1)^{\N}$ and define
  $(y_n)$ by $y_{2k}=z_k$ and $y_{2k+1}=x_k$ for all $k$. Then let
  $F(\bar x)=K(\bar y)$ as defined by \eqref{eq:defknot}. 
  Suppose $(x_n)$ and $(x_n')$ are $E^*$-equivalent sequences.
  Let $(y_n)$ and $(y_n')$ be the corresponding sequences as above.
  Now $(x'_n-x_n)\stackrel{n\to \infty}{\longrightarrow}0$ and of course also
  $(y_n'-y_n)\stackrel{n\to \infty}{\longrightarrow}0$.
  So by Proposition~\ref{prop:EquivalentKnots} $K(\bar y_n')$ and 
  $K(\bar y_n)$ are equivalent.
  Suppose that $(x_n'-x_n)$ does not converge to zero. Then
  $(y_n'-y_n)$ does not converge at all. Taking the even numbers as $A$,
  Proposition~\ref{prop:NonEquivalentKnots} implies that $K(\bar y'_n)$ and $K(\bar y_n)$
  are not equivalent.

  Let us show that the reduction is continuous. 
  Let $\bar x$ be a sequence and $\e>0$ and let us find a neighborhood $U$ of $\bar x$
  such that for all $\bar x'\in U$ we have $d(K(\bar x),K(\bar x'))<\e$.
  Let $k$ be so big that $21\e_k<\e$. For $n\le k$ let
  Let
  \begin{equation}
    \d_n=\frac{\e_k}{3\cdot 2^{n}(k+1)}.\label{eq:defofdelta}
  \end{equation}
  Let $U$ be the neighborhood 
  $$U=\prod_{n=0}^k B_{S^1}(x_n,\d_n)\times \prod_{n> k} S^1.$$
  For $m<n$ denote $\hat g^{\bar x}_{m,n}=\comp_{i=m}^{n}\theta^{x_i}g_i \theta^{-x_i}(z)$
  (cf.~\ref{eq:rotationform}).
  Now for all $z\in T_{\e_{k+1}}$ and $\bar x'\in U$ we have
  \begin{eqnarray*}
    d(\hat g^{\bar x}_k(z),\,\hat g^{\bar x'}_k(z))
    &=&d\big((\theta^{x_0}  g_0\theta^{-x_0} \circ \hat g^{\bar x}_{1,k}) (z),\,
             (\theta^{x'_0} g_0\theta^{-x'_0}\circ \hat g^{\bar x'}_{1,k})(z)\big)\\
    &=&d\big((\theta^{x'_0-x_0}  g_0\theta^{-x_0} \circ \hat g^{\bar x}_{1,k})(z),\,
                                (g_0\theta^{-x'_0}\circ \hat g^{\bar x'}_{1,k})(z)\big)\\
    &\le&|x_0'-x_0|+d\big((g_0\theta^{-x_0} \circ \hat g^{\bar x}_{1,k})(z),\,
                          (g_0\theta^{-x'_0}\circ \hat g^{\bar x'}_{1,k})(z)\big).
  \end{eqnarray*}
  The last inequality holds for a similar reason as~\eqref{eq:HomeoLip}. The first equality
  is by the definition \eqref{eq:rotationform} and the second follows from that $\theta^{s}$ is
  an isometry. Continuing the chain of inequalities, using the facts that
  $|x_0'-x_0|<\d_0$ and $g_0$ is $L_0$-Lipschitz and $L_0<2$, (by \eqref{sequences(g)}, 
  \eqref{sequences(e)}, Proposition~\ref{prop:BilipPathMetric} and the discussion
  after it), we have
  \begin{eqnarray*}
    &\le&\d_0+2d\big((\theta^{-x_0} \circ \hat g^{\bar x}_{1,k})(z),\,
                     (\theta^{-x'_0}\circ \hat g^{\bar x'}_{1,k})(z)\big)\\
    &=&\d_0+2d\big((\theta^{x_0'-x_0} \circ \hat g^{\bar x}_{1,k})(z),\,
                                           \hat g^{\bar x'}_{1,k}(z)\big)\\
    &\le&\d_0+2(\d_0+d\big(\hat g^{\bar x}_{1,k}(z)\,
                           \hat g^{\bar x'}_{1,k}(z))\big)\\
    &=&3\d_0+2d\big(\hat g^{\bar x}_{1,k}(z),\,
                    \hat g^{\bar x'}_{1,k}(z))\big).
  \end{eqnarray*}
  Continuing by induction in the exact same way we get
  \begin{eqnarray*}
    &=&3\d_0+2d(\hat g^{\bar x}_{1,k}(z),
                \hat g^{\bar x'}_{1,k}(z)))\\
    &\le&3\d_0+2(3\d_1 + 2d(\hat g^{\bar x}_{2,k}(z),
                         \hat g^{\bar x'}_{2,k}(z)))\\
    &\le&3\d_0+3\cdot 2\d_1 + 2^2d(\hat g^{\bar x}_{2,k}(z),
                               \hat g^{\bar x'}_{2,k}(z)))\\
    &\vdots& \vdots \\
    &\le &3\d_0+3\cdot 2\d_1 + 3\cdot 2^2\d_2 +\cdots +3\cdot 2^{k}\d_k\\
    &=&\sum_{n=0}^k 3\cdot 2^n \d_n.
  \end{eqnarray*}
  By the definition of $\d_n$, \eqref{eq:defofdelta}, this equals to 
  $$\sum_{n=0}^k \frac{\e_k}{k+1}=\e_k.$$
  Thus, we get that
  for all $z\in T_{\e_{k+1}}$ and $\bar x'\in U$ we have
  $$d(\hat g^{\bar x}_k(z),\,\hat g^{\bar x'}_k(z))\le \e_k$$
  and in particular for all $s\in S^1$ and and $\bar x'\in U$ we have
  $$d(\hat f^{\bar x}_k(s),\hat f^{\bar x'}_k(s))\le \e_k.$$  
  From the proof of Proposition~\ref{prop:DefOfKnot}, we know that in the sup-metric
  $$d(\hat f^{\bar x}_k,f^{\bar x})\le 10\e_{k}\ \text{ and }\ d(\hat f^{\bar x'}_k,f^{\bar x'})\le 10\e_{k}$$
  so from the above and the choice of $k$, we have 
  $$d(f^{\bar x},f^{\bar x'})\le 21\e_{k}<\e$$
  in the sup metric. But this of course implies that the Hausdorff distance of
  $K(\bar x)=\Im (f^{\bar x})$ and $K(\bar x')=\Im (f^{\bar x'})$ is at most $\e$
  as well. Since the function $\bar x\mapsto \bar y$ in the definition of $F(\bar x)$
  is continuous on $(S^1)^\N$, the reduction $F$ is continuous.
\end{proof}

\begin{Remark}
  As remarked after \eqref{eq:term}, all
  knots in the above construction have homeomorphic complements. This
  strengthens the side remark of \cite{Nanyes} that there are
  uncountably many non-equivalent knots with homeomorphic complements
  to that there is a \emph{non-classifiable} uncountable set of knots
  all of which have homeomorphic complements.
\end{Remark}

\bibliography{ref1}{}
\bibliographystyle{alpha}

\end{document}